\documentclass[oneside]{amsart}
\usepackage{amsmath,amssymb}
\usepackage{graphicx}
\usepackage[cmtip,all]{xy}
\usepackage{geometry}
\usepackage{placeins}
\usepackage[latin1]{inputenc}
\usepackage{dsfont}
\usepackage{amssymb,amsmath}
\usepackage{verbatim}
\usepackage{color}
\usepackage{enumerate}
\usepackage{url}
\usepackage{ifthen}
\usepackage{tikz}
\usetikzlibrary{arrows}
\usetikzlibrary{decorations.markings}
\usepackage{booktabs}
\usepackage{array}

\newtheorem{theorem}{Theorem}[section]
\newtheorem{lemma}[theorem]{Lemma}
\newtheorem{proposition}[theorem]{Proposition}
\newtheorem{conjecture}[theorem]{Conjecture}
\newtheorem{corollary}[theorem]{Corollary}

\theoremstyle{definition}
\newtheorem{definition}[theorem]{Definition}

\theoremstyle{remark}
\newtheorem{remark}[theorem]{Remark}

\numberwithin{equation}{section}

\DeclareMathOperator{\Fil}{Fil}
\newcommand\Z{\ensuremath{\mathbb Z}}
\newcommand\Q{\ensuremath{\mathbb Q}}
\newcommand\F{\ensuremath{\mathbb F}}

\newcommand\Div{\operatorname{Div}}

\newcommand\GL{\operatorname{GL}}
\newcommand\Hom{\operatorname{Hom}}

\newcommand\im{\operatorname{im}}

\newcommand\M{\operatorname{M}}

\newcommand\ord{\operatorname{ord}}

\newcommand\PGL{\operatorname{PGL}}

\newcommand\SL{\operatorname{SL}}

\newcommand{\cH}{\mathcal{H}}
\newcommand{\cD}{\mathcal{D}}

\newcommand{\cT}{\mathcal{T}}

\newcommand{\cU}{{\mathcal{U}}}
\newcommand{\cV}{{\mathcal{V}}}
\newcommand{\cE}{{\mathcal{E}}}
\newcommand{\cF}{{\mathcal{F}}}

\newcommand{\meas}{\operatorname{Meas}}
\newcommand{\hc}{\cF_{\text{har}}}

\newcommand{\tto}[1]{%
\ifthenelse{\equal{#1}{}}{\to}{\stackrel{#1}{\to}}}

\def\cO{{\mathcal O}}

\newcommand{\mtx}[4]{\left(\begin{matrix}#1&#2\\#3&#4\end{matrix}\right)}

\newcommand{\smtx}[4]{\left(\begin{smallmatrix}#1&#2\\#3&#4\end{smallmatrix}\right)}

\newcommand{\tns}{\otimes}

\def\M{\operatorname{M}}

\def\P{\mathbb P}
\newcommand{\comp}{\begin{picture}(6,5)(-3,-2)\put(0,1){\circle{2}} \end{picture}}\def\circ{\comp}

\def\cY{\mathcal Y}
\def\cB{\mathcal B}
\def\ol{\overline}

\newcommand{\ra}{\rightarrow}

\newcommand{\lra}{\longrightarrow}

\newcommand{\QQ}{\Q}
\newcommand{\PP}{\mathbb{P}}
\newcommand{\ZZ}{\Z}

\newcommand{\GM}{\Gamma^D_0(M)}
\newcommand{\GMp}{\Gamma^D_0(pM)}

\newcommand{\GMhat}{\hat{\Gamma}^D_0(M)}

\newcommand{\opU}{U}
\newcommand{\opT}{T}
\newcommand{\opW}{W}

\def\Xint#1{\mathchoice
{\XXint\displaystyle\textstyle{#1}}%
{\XXint\textstyle\scriptstyle{#1}}%
{\XXint\scriptstyle\scriptscriptstyle{#1}}%
{\XXint\scriptscriptstyle\scriptscriptstyle{#1}}%
\!\int}
\def\XXint#1#2#3{{\setbox0=\hbox{$#1{#2#3}{\int}$}
\vcenter{\hbox{$#2#3$}}\kern-.5\wd0}}

\newcommand{\charf}{{\mathds{1}}}

\begin{document}

\title[Overconvergent cohomology and Darmon points]{Overconvergent cohomology \\ and quaternionic Darmon points}

\date{\today}

\author{Xavier Guitart}
\address{Institut f\"ur Experimentelle Mathematik, Essen, Germany}
\curraddr{}
\email{xevi.guitart@gmail.com}

\author{Marc Masdeu}
\address{University of Warwick, Coventry, United Kingdom}
\curraddr{}
\email{m.masdeu@warwick.ac.uk}


\maketitle
\begin{abstract}
We develop the (co)homological tools that make effective the construction of the quaternionic Darmon points introduced by Matthew Greenberg. In addition, we use the overconvergent cohomology techniques of Pollack--Pollack to allow for the efficient calculation of such points. Finally, we provide the first numerical evidence supporting the conjectures on their rationality.
\end{abstract}

\section{Introduction}
Let $E$ be an elliptic curve over $\Q$ of conductor $N$ and let $p$ be a prime dividing $N$ exactly.
Consider a factorization of the form $N=pDM$, with $D$ the product of an even (possibly zero) number of
distinct primes and  $(D,M)=1$. Let $K$ be a real quadratic field in which all primes dividing $M$ are split, and all primes 
dividing $pD$ are inert. Denote by $\cH_p=K_p\setminus \Q_p$ the $K_p$-points of the $p$-adic upper half plane. 

In the case $D=1$, Darmon introduced in the seminal article \cite{darmon-integration} a construction of local points
$P_\tau\in E(K_p)$ associated to elements $\tau\in K\cap\cH_p $, defined as certain Coleman
integrals of the modular form attached to $E$. He conjectured these points to be rational over certain ring class
fields
of $K$, and to behave in many aspects as the classical Heegner points arising from quadratic imaginary fields. A proof of these conjectures would certainly shed new light on new instances of the Birch--Swinnerton-Dyer conjecture. The reader can consult~\cite[Section 5]{darmon-integration} and~\cite[Section 4]{darmon-green} for a discussion of this circle of ideas.

These
conjectures 
are supported by some partial theoretical results such as \cite{bertolini-darmon-genus}, but at the moment the main evidence comes from 
explicit numerical computations. Darmon--Green \cite{darmon-green} provided the first systematic algorithm and 
numerical calculations for curves satisfying the additional restriction that $M=1$. Using overconvergent 
methods in the evaluation of the integrals Darmon--Pollack \cite{darmon-pollack} were able to give a 
much faster algorithm, which  in practice can be used
(assuming Darmon's conjectures) as an efficient method for computing algebraic points of infinite order on 
$E(K^{\text{ab}})$. The restriction $M=1$ in these algorithms was dispensed with in \cite{dphc}, which
allowed to provide numerical evidence for curves of non-prime conductor.

In the case $D>1$, Greenberg \cite{Gr} proposed a construction of Darmon-like points in $E(K_p)$, by means of certain 
$p$-adic integrals related to modular forms on quaternion division algebras of discriminant $D$. He also conjectured 
that these points behave in many aspects as Heegner points and, in 
particular, that they are rational over ring class fields of $K$.

 Greenberg's conjecture was motivated 
by the analogy with \cite{darmon-integration}, but up to now there was no numerical evidence of the rationality of such 
points in the quaternionic case $D>1$. In fact, as Greenberg points out in~\cite[Section 12]{Gr}, the lack of sufficiently developed algorithms for computing in the cohomology of arithmetic groups has prevented the finding of any such evidence. The main purpose of the present work is precisely to provide an explicit algorithm that allows for the effective computation of the quaternionic $p$-adic Darmon points introduced by Greenberg.

Actually, the aim of the article is threefold. Firstly, we develop the (co)homological methods that make effective the construction of~\cite{Gr}. Secondly, we relate the $p$-adic integrals that appear in the construction to certain overconvergent cohomology classes, in order to derive an efficient algorithm for the computation of the quaternionic Darmon points. Finally, we gather extensive evidence supporting the rationality conjectures of~\cite{Gr}.

In order to describe more precisely the contents of the article, it is useful to briefly recall the structure of Greenberg's construction (a more complete and detailed account will be given in Section~\ref{subsection: Greenberg construction}). Let $B/\Q$ be the indefinite quaternion algebra of discriminant $D$. Let also $R_0(M)\subset B$ be an Eichler order of level $M$, and denote by $\Gamma$ the group of reduced norm $1$ units in $R_0(M)\otimes_\Z\Z[1/p]$. The construction of the point $P_\tau\in E(K_p)$ can be divided into three stages:
\begin{enumerate}
\item The construction a certain cohomology class $\mu_E\in H^1(\Gamma,\text{Meas}(\mathbb{P}^1(\Q_p),\Z))$ canonically attached to $E$, where $\text{Meas}(\P^1(\Q_p),\Z)$ denotes the $\Z$-valued measures of $\P^1(\Q_p)$;
\item The construction of a homology class $c_\tau\in H_1(\Gamma,\Div^0\cH_p)$, associated to the element $\tau\in \cH_p$; and
\item Finally, the construction of a natural $K_p^\times$-valued integration pairing $\Xint\times\langle\, , \, \rangle $ between the above cohomology and homology groups.
\end{enumerate}
The point $P_\tau$ is then defined as the image under Tate's isomorphism $K_p^\times/\langle q_E\rangle \simeq E(K_p)$ of the quantity $J_\tau:=\Xint\times \langle c_\tau,\mu_E\rangle \in K_p^\times$.

Section~\ref{subsection: 
Background} is devoted to background material and to fix certain choices on the (co)homology groups that will be useful in our algorithms, and in Section~\ref{subsection: Greenberg construction} we give a more detailed description of the construction of Greenberg. The main contributions of this work are presented in Sections~\ref{sec: the algorithm}, \ref{sec: overconvergent}, and \ref{sec: numerical evidence}.

In Section~\ref{sec: the algorithm} we provide algorithms for computing the homology class $c_\tau$ and the cohomology class $\mu_E$. That is to say, we give explicit methods for working with the (co)homology groups arising in the construction, which allow for the effective numerical calculation of $\mu_E$ and $c_\tau$ in concrete examples. This already gives rise to an algorithm for the calculation of the point $P_\tau$, since the integration pairing $\Xint\times\langle c_\tau ,\mu_E\rangle $ can then be computed by the well known method of Riemann products. We end Section~\ref{sec: the algorithm} with a detailed concrete calculation of a Darmon point $P_\tau$ by means of this algorithm. 

Although the method of Riemann products is completely explicit and can be used in principle to evaluate the integration pairing, it has the drawback of being computationally inefficient. In fact, its running time depends exponentially on the number of $p$-adic digits of accuracy to which the output is desired. This is the problem that we address in Section~\ref{sec: overconvergent}, in which we give an efficient, polynomial-time, algorithm for computing the integration pairing $\Xint\times\langle \, , \, \rangle $. This method is  based on the overconvergent cohomology lifting theorems of~\cite{pollack-pollack}, and can be seen as a generalization to the quaternionic setting of the overconvergent modular symbols method of~\cite{darmon-pollack}. Used in conjunction  with the algorithms of Section~\ref{sec: the algorithm} for the homology and cohomology classes, it provides an efficient algorithm for computing the quaternionic Darmon points.

Finally, in Section~\ref{sec: numerical evidence} we provide extensive calculations and numerical evidence in support of the conjectured rationality of Greenberg's Darmon points, which were computed using an implementation in Sage (\cite{sage}) and Magma (\cite{magma}) of the algorithms described in Sections~\ref{sec: the algorithm} and \ref{sec: overconvergent}.

\subsection*{Acknowledgments} We are grateful to Victor Rotger for suggesting the
problem, as well as to Henri Darmon, Matthew Greenberg, Matteo Longo, Robert Pollack, Eric Urban, and John Voight for valuable 
exchanges and suggestions. We also wish to express our gratitude to the anonymous referee, whose comments encouraged us to strengthen some of the results. Guitart wants to thank the Max Planck
Institute for Mathematics and the Hausdorff Research Institute for Mathematics. The work in this article was carried out while Masdeu was at Columbia University as a Ritt assistant professor. Both authors were partially supported by MTM2009-13060-C02-01 and 2009 SGR 1220, and Guitart was also partially supported by SFB/TR~45.

\subsection*{Notation} The following notation shall be in force throughout the article. Let $E$ be an elliptic curve over
$\Q$ of conductor $N$  and let $p$ be a prime dividing $N$ exactly. The conductor is factored as $N=pDM$, where $D>1$
is the product of an even number of distinct primes and $M$ and $D$ are relatively prime. Let $K$ be a real
quadratic field in which all primes dividing $M$ are split and all primes dividing $pD$ are inert, and let $\cO_K$ be the
ring of integers of $K$.

Let $B$ be the quaternion algebra over $\Q$ of discriminant $D$. For every $\ell\mid pM$ we fix an algebra
isomorphism
\[
\iota_\ell\colon B\otimes_\Q
\Q_\ell\stackrel{\simeq}{\lra} \M_2(\Q_\ell).
\]
Let $R_0(M)\subset B$ be an Eichler order of level $M$ such that for every $\ell\mid M$  
\begin{align}\label{eq: Eichler order locally}
 \iota_\ell(R_0(M))=\left\{\smtx a b c d \in \M_2(\Z_\ell) \colon c\equiv 0 \pmod \ell\right\}.
\end{align}
Similarly, let
$R_0(pM)\subset R_0(M)$ be an Eichler order of level $pM$ that satisfies \eqref{eq: Eichler order locally} also for
$\ell=p$. Denote by $\GM=R_0(M)^\times_1\ \text{ and }\ \GMp=R_0(pM)_1^\times$ their group of reduced norm $1$ units. Finally,
let
\[R=R_0(M)\otimes_\Z\Z[1/p]\ \text{ and } \ \Gamma=R_1^\times=\left\{ \gamma\in R\colon 
\operatorname{nrd}(\gamma)=1 \right\}.\]

\section{Preliminaries on Hecke operators, the Bruhat--Tits tree, and measures}\label{subsection: 
Background}
All the material in this section is well-known. We present it particularized to our setting and we fix 
certain choices that will be important especially in Section \ref{sec: overconvergent
cohomology}.

\subsection{Hecke operators on homology and cohomology}
We recall first some well-known facts on group (co)homology which can all be found for example in~\cite{brown}. This will also fix the notation to be used in the sequel.

Let $G$ be a group and $V$ a commutative left
$G$-module. The groups
of \emph{$1$-chains}
and \emph{$2$-chains} are defined, respectively, as
\[
 C_1(G,V)=\Z[G]\otimes_\Z V , \ \ C_2(G,V)=\Z[G]\otimes_\Z \Z[G]\otimes_\Z V.
\]
 The \emph{boundary maps}
are induced by the formulas, for $g$ and $h$ in $G$ and $v\in V$,
\begin{align}\label{eq: boundary maps}
 \partial_1(g\otimes v)&= gv-v;& \partial_2(g\otimes h\otimes v)&=h\otimes g^{-1}v-gh\otimes v+ g\otimes v.
\end{align}

We denote by $Z_1(G,V)=\ker\partial_1$ the group of $1$-cycles, by $B_1(G,V)=\im\partial_2$ the group of $1$-boundaries, and by $H_1(G,V)=Z_1/B_1$ the first homology group of $G$ with coefficients in $V$.


Dually, one defines the group of \emph{$1$-cochains} $C^1(G,V)$, the group of \emph{$1$-coboundaries} $B^1(G,V)$, the group of $1$-cocycles $Z^1(G,V)$ and the first cohomology group $H^1(G,V)=Z^1/B^1$.


We are mainly interested in the (co)homology of the group $G=\GMp$. Consider also the semigroup $\Sigma_0(pM)$ defined as
\begin{align}\label{eq: def of simga_0(PM)}
 \Sigma_0(pM)=B^\times\cap \prod_\ell \Sigma_\ell, \ \text{where }
\end{align}
\[
 \Sigma_\ell = \begin{cases} \text{ the set of elements
in $R_0(pM)$ with non-zero norm}, & \mbox{if } \ell\nmid pM; \\ \iota_\ell^{-1}\left(\{\smtx a b c d\in
\M_2(\Z_\ell)\colon c\equiv 0\pmod \ell, \ d\in\Z_\ell^\times,\
ad-bc\neq 0\}\right), & \mbox{if } \ell\mid pM.
\end{cases}
\]
Suppose that the $\GMp$-action on $V$ extends to an action of the semigroup $\Sigma_0(pM)$. Then there are natural Hecke operators acting on $H_1(\Gamma_0^D(pM),V)$ and $H^1(\Gamma_0^D(pM),V)$ whose definition we proceed to recall, following~\cite{ash-stevens}.

\subsubsection*{The operators $\opT_\ell$ and $\opU_\ell$}

Let $\ell$ be a prime not dividing $D$, and let $g(\ell)\in \Sigma_0(pM)$ be an element of reduced norm
$\ell$. The double coset $\GMp g(\ell)\GMp$ decomposes as a finite disjoint union of right $\GMp$-cosets: 
\begin{align}\label{eq: double cosets}
\GMp g(\ell)\GMp=\bigsqcup_{i\in I_\ell}g_i\GMp,
\end{align}
for certain $g_i\in \Sigma_0(pM)$ of reduced norm $\ell$. The number of cosets in \eqref{eq: double cosets}, i.e., the
cardinal of $I_\ell$, is
$\ell+1$ if $\ell\nmid pM$ and $\ell$ otherwise. Let $t_i:\GMp\ra \GMp$ be the map defined by the equation
\[\gamma^{-1}g_i=g_{\gamma\cdot i}t_i(\gamma)^{-1}, \text{ for some index } \gamma\cdot i \in I_\ell.\] We
remark that $i\mapsto \gamma\cdot i$ is a permutation of $I_\ell$. Decomposition \eqref{eq: double cosets} induces
maps $T_\ell$ on $1$-chains and $1$-cochains
as follows:  for a chain
$c=\sum_{g}g\otimes v_g\in C_1(\GMp,V)$ and a cochain $f\in C^1(\GMp,V)$ then
\begin{align}\label{eq: def of T_ell}
 \opT_\ell c&=\sum_{i\in I_\ell}\sum_g t_i(g)\otimes g_i^{-1}v_g;&
 (\opT_\ell f)(g)&=\sum_{i\in I_\ell} g_if(t_i(g)).
\end{align}
The map $\opT_\ell$ on chains (resp. cochains) respects cycles and boundaries (resp. cocycles and coboundaries). The
Hecke operators are the induced endomorphisms on homology and cohomology
which do not
depend neither on the choice of $g(\ell)$
nor on the representatives $g_i$ of  $\eqref{eq: double cosets}$. Following the usual notational conventions if 
$\ell\mid pM$ we set $\opU_\ell=\opT_\ell$. We remark that the operators $\opT_\ell$ and $\opU_\ell$ on homology and cohomology are independent of the choices made in the definition.
However,
 as maps on chains and cochains (and even as maps on cycles and cocycles) they do
depend on these choices. 

In \S \ref{sec: overconvergent} it will be important to work
with the $\opU_p$-operator on cochains obtained by means of a specific decomposition \eqref{eq: double
cosets} which we now describe. In order to do so, we next fix a choice of certain elements of
$\Sigma_0(pM)$; these elements (and the notation for them) shall be in force
for the rest of the article.
\begin{itemize}
 \item Let $\Upsilon=\{\gamma_0,\dots,\gamma_{p}\}$ be a system of representatives for 
$\GMp\backslash\GM $ satisfying that
\begin{align}\label{eq: def of gammas}
 \gamma_0=1, \ \ \text{and for $i>0$ }\  \iota_p(\gamma_i)=u_i\smtx{0}{-1}{1}{i},
\end{align}
for some  $u_i$ belonging to
\[
 \Gamma_0^{\text{loc}}(p)=\{\smtx a b c d\in \SL_2(\Z_p) \colon c\equiv 0 \pmod p\}.
\]
 \item Let $\omega_p\in R_0(pM)$ be an element that normalizes $\GMp$ and such that
\begin{align}\label{eq: w_p}
 \iota_p(\omega_p)=u'\smtx{0}{-1}{p}{0},\ \text{ for some } u'\in \Gamma_0^{\text{loc}}(p).
\end{align}
Also, let $\omega_\infty\in R(pM)$ be an element of reduced norm $-1$  that  normalizes $\GMp$.

\item Finally, set \begin{align}\label{eq: def of s_i}\pi=\gamma_p^{-1}\omega_p \ \ \text{ and }
s_i=\gamma_i^{-1}\omega_p^{-1}, \text{ for } i=1,\dots, p.\end{align}
We remark that 
\begin{align}\label{eq: local s_i}
 \iota_p(\pi)=\smtx{p}{0}{0}{1}u_\pi \text{ and } \iota_p(s_i)=\smtx{p}{-i}{0}{1}u_i'
\end{align}
for some $u_\pi,u_i'\in \Gamma_0^{\text{loc}}(p)$.
\end{itemize}
Observe that $\pi\in\Sigma_0(pM)$ has reduced norm $p$; we will work with the Hecke operator on cycles and
cocycles associated 
to the double coset $\GMp \pi \GMp$. One checks that the $s_i$ defined above decompose it into right 
cosets, namely
\begin{align}\label{eq: double coset U_p}
 \GMp \pi \GMp = \bigsqcup_{i=1}^p s_i\GMp.
\end{align}
Then $t_i:\GMp\ra \GMp$ is the function defined by 
\begin{align}\label{eq: def of t_k}\gamma^{-1}s_i=s_{\gamma\cdot i}t_i(\gamma)^{-1}, \text{ for certain index }
\gamma\cdot i \in\{1,\dots,p\}.\end{align} For
$c=\sum_{g}g\otimes v_g\in Z_1(\GMp,V)$ and $f\in Z^1(G,V)$ formulas \eqref{eq: def of T_ell} particularize to
\begin{align}\label{eq: def U_p}
 \opU_p c=\sum_{i=1}^p\sum_g t_i(g)\otimes s_i^{-1}v_g; \ \ 
 (\opU_p f)(g)=\sum_{i=1}^p s_if(t_i(g)).
\end{align}

\subsubsection*{Atkin--Lehner involutions} The Atkin-Lehner involutions at $p$ on
cycles and
cocycles are given by the formulas:
\[
 \opW_p c=\sum_g \omega_p^{-1} g \omega_p\otimes \omega_p^{-1}v_g; \ \ 
 (\opW_p f)(g)= \omega_pf(\omega_p^{-1} g\omega_p).
\]
Similarly, Atkin--Lehner involutions at infinity are defined as:
\[
 \opW_\infty c=\sum_g \omega_\infty^{-1} g \omega_\infty\otimes \omega_\infty^{-1}v_g; \ \ 
 (\opW_\infty f)(g)= \omega_\infty f(\omega_\infty^{-1} g\omega_\infty).
\]
These formulas induce well-defined involutions on the homology $H_1(\GMp,V)$ and on the cohomology $H^1(\GMp,V)$.

\subsubsection*{Hecke algebras}Let $[\opT_\ell]$, $[\opU_\ell]$ and $[\opW_\infty]$ be
formal variables. If $m\in\Z_{>0}$ we denote by $\mathbb{T}^{(m)}$ the Hecke algebra ``away from $m$''; i.e., the
$\Z$-algebra generated by $[W_\infty]$ and by the $[T_\ell]$ and $[U_\ell]$ with $\ell\nmid m$.
Since the Hecke operators commute with each other $\mathbb{T}^{(m)}$ acts on $H_1(\GMp,V)$ and
$H^1(\GMp,V)$ by letting each formal variable act as the corresponding Hecke operator. 

If $\lambda: \mathbb{T}^{(m)}\ra \Z$ is a ring homomorphism and $H$ is a $\mathbb{T}^{(m)}$-module let
\[
 H^\lambda=\{x\in H \colon tm=\lambda(t)x \text{ for all } t\in \mathbb{T}^{(m)}\}.
\]
The \emph{degree character} $\deg:\mathbb{T}^{(pD)}\ra \Z$ is defined by
\[
 \deg[\opT_\ell]=\ell+1, \ \deg[\opU_\ell]=\ell, \ \deg \opW_\infty=1.
\]
Thanks to the modularity theorem of~\cite{Wi},~\cite{MR1839918}  the elliptic curve $E/\Q$ defines two characters $\lambda_E^+,\lambda_E^-:\mathbb{T}^{(pD)}\ra \Z$ by
\begin{align}\label{eq: def of lambda}
 \lambda_E^\pm[\opT_\ell]=\ell+1-|E(\F_\ell)|, \ \lambda_E^\pm[\opU_\ell]=\ell+1-|E(\F_\ell)|, \
\lambda_E^\pm[\opW_\infty]=\pm 1.
\end{align}

\begin{remark}
There are also Hecke operators acting on $H_1(\Gamma,V)$ and $H^1(\Gamma,V)$. They are defined similarly, but using 
double cosets of the form $\Gamma g'(\ell) \Gamma$ (this time $g'(\ell)$ 
is an element of $R$ of reduced norm $\ell$); see, e.g., \cite[\S2]{LRV} for more details. 
For our purposes it is enough to say that for $\ell\nmid pM$ one can choose $g(\ell)\in R_0(pM)$ and $g'(\ell)\in R$ 
elements of reduced norm $\ell$ such that the decompositions
\begin{align*}
 \GMp g(\ell)\GMp=\bigsqcup_{i=0}^\ell g_i\GMp \ \text{ and } \Gamma g(\ell)\Gamma=\bigsqcup_{i=0}^\ell g_i\Gamma 
\end{align*}
hold with the same choice of $g_i\in \Sigma_0(pM)$. Thus formulas \eqref{eq: 
def of T_ell} also give the $T_\ell$ operator on $H_1(\Gamma,V)$ and $H^1(\Gamma,V)$ in this case.
\end{remark}

\subsection{The Bruhat-Tits Tree}

Let $\cT$ be the Bruhat--Tits tree of
$\PGL_2(\Q_p)$ and denote by $\cV$ its set of vertices and by  $\cE$ its set of (directed) edges. It is well known that $\cT$ is a 
$(p+1)$-regular 
tree. In addition  $\cV$ can be identified with the set of homothety classes of $\Z_p$-lattices in $\Q_p^2$, and directed
edges with ordered pairs of vertices $(v_1,v_2)$ such that $v_1$ and $v_2$ can be represented by lattices $\Lambda_1$,
$\Lambda_2$ with $p\Lambda_1\subsetneq\Lambda_2\subsetneq \Lambda_1$. For $e=(v_1,v_2)\in \cE$ we denote by $s(e)=v_1$
the source of $e$, by $t(e)=v_1$ its target and by $\bar e=(v_2,v_1)$ its opposite.

Let $v_*$ be the vertex represented by $\Z_p^2$, let $\hat v_*$ be the one represented by
$\Z_p\oplus p\Z_p$, and let $e_*$ be the edge $(v_*,\hat v_*)$. A vertex $v$ is said to be even (resp. odd) if its distance $d(v,v_*)$ to
$v_*$ is even (resp. odd), and $e\in\cE$ is said to be even (resp. odd) if $s(e)$ is even (resp. odd). We
denote by $\cV^+$ (resp. $\cV^-$) the set of even (resp. odd) vertices and by $\cE^+$ (resp. $\cE^-$) the set of even
(resp. odd) edges.

The group $\GL_2(\Q_p)$ acts on $\Q_p$ by fractional linear transformations
\[
 g\tau=\frac{a\tau+b}{c\tau+d}, \quad\text{for } g=\smtx a b c d \in \GL_2(\Q_p)\text{ and }\tau\in\Q_p.
\]
This induces an action of $\GL_2(\Q_p)$ on $\Z_p$-lattices, which gives rise to an action of $\GL_2(\Q_p)$ on  $\cV$ that preserves distance, thus inducing an action on $\cT$ and on $\cE$.

We can make $\Gamma$ act on $\cT$ by means of the fixed isomorphism
\[
 \iota_p \colon B\otimes\Q_p \lra \M_2(\Q_p).
\]
We denote this action simply as $g(v)$ and $g(e)$, for $g\in \Gamma$ and $v\in\cV$, $e\in\cE$. Strong approximation, 
using the fact that $B$ is unramified at infinity, implies that $\Gamma$ acts transitively on $\cE^+$. A fundamental 
domain (in the sense of~\cite[\S4.1]{serre-trees}) for this action is given by
\[
\xymatrix@R5pt{
v_* &e_*& \hat v_*\\
*=0{\bullet}\ar@{-*=0@{>}}[r]&*=0{}\ar@{-}[r]&*=0{\bullet}
}
\]
Moreover, we have:
\begin{enumerate}
 \item $\operatorname{Stab}_{\Gamma}(v_*)=\GM$, and $\operatorname{Stab}_\Gamma(\hat v_*)= \GMhat:=\omega_p^{-1}\GM\omega_p$.
\item $\operatorname{Stab}_\Gamma(e_*)=\GMp$.
\end{enumerate}
This implies that $\Gamma=\GM \star_{\GMp} \GMhat$, where $\star$ denotes ``amalgamated product''.

In particular, the maps $g\mapsto g^{-1}(e_*)$ and $g\mapsto g^{-1}(v_*)$ induce bijections
\[
 \GMp\backslash\Gamma\stackrel{1:1}{\longleftrightarrow}\cE^+ ,\quad
\GM\backslash\Gamma \stackrel{1:1}{\longleftrightarrow}  \cV^+.
\]
In the following we will fix a convenient system of coset representatives of $\GMp\backslash\Gamma$ indexed by the even edges, and another system of coset representatives of $\GM\backslash\Gamma$ indexed by the even vertices. These were introduced in~\cite[Definition 4.7]{LRV} and called \emph{radial systems}.

Recall $\Upsilon=\{\gamma_0=1,\gamma_1,\dots,\gamma_p\}$ the set of representatives for $\GMp\backslash\GM$
we fixed in \eqref{eq: def of gammas}. Define $\tilde\gamma_0=1$ and, for $i=1,\ldots,p$, define $\tilde\gamma_i=p^{-1}\omega_p\gamma_i\omega_p$.
\begin{lemma}
We have
\[
\GMp\backslash \GMhat = \coprod_{i=0}^p \GMp \tilde \gamma_i.
\]
\end{lemma}
\begin{proof}
Clearly the set $\{1,\omega_p^{-1}\gamma_1\omega_p,\dots,\omega_p^{-1}\gamma_p\omega_p\}$ is a system of 
representatives for the quotient $\GMp\backslash\GMhat$. The elements $p^{-1}\omega_p\gamma_i\omega_p $ and
$\omega_p^{-1}\gamma_i\omega_p$ belong to the same coset modulo $\GMp$. Indeed, this follows from the identity
\[
 p^{-1}\omega_p\gamma_i\omega_p =p^{-1}\omega_p^{2}\omega_p^{-1}\gamma_i\omega_p
\]
and the fact that  $p^{-1}\omega_p^2\in\GMp$ (see, e.g., \cite[\S3.2]{Gr}).
\end{proof}

\begin{definition}\label{def: radial system}
Define $\{\gamma_e\}_{e\in\cE^+}$ and $\{\gamma_v\}_{v\in\cV}$ to be the systems of representatives respectively for
$\GMp\backslash\Gamma$ and $\GM\backslash\Gamma$ uniquely determined by
the conditions:
\begin{enumerate}
\item\label{item: rs 1} $\gamma_{v_*}=\gamma_{\hat v_*}=1$;
 \item\label{item: rs 2} $\{\gamma_e\}_{s(e)=v}=\{\gamma_i\gamma_v\}_{i=0}^p$ for all $v\in\cV^+$;
\item\label{item: rs 3}  $\{\gamma_e\}_{t(e)=v}=\{\tilde\gamma_i\gamma_v\}_{i=0}^p$ for all $v\in\cV^-$;
\item\label{item: rs 4}  $\gamma_{s(e)} = \gamma_e$ for all $e\in\cE^+$ such that $d(t(e),v_*)< d(s(e),v_*)$;
\item\label{item: rs 5}  $\gamma_{t(e)} = \gamma_e$ for all $e\in\cE^+$ such that $d(t(e),v_*)>d(s(e),v_*)$.
\end{enumerate}
\end{definition}
By construction $\cY=\{ \gamma_e \}_{e\in \cE^+}$ is a radial system. Indeed, by definition a radial system is one
satisfying conditions 
\ref{item: rs 1}, \ref{item: rs 2}, and \ref{item: rs 3} above for some set of representatives for 
$\GMp\backslash\GM$ and $\GMp\backslash\GMhat$. What we have done is to fix a choice of radial system by choosing
$\{\gamma_0,\dots,\gamma_p\}$ and $\{\tilde\gamma_0,\dots,\tilde\gamma_p\}$ as such representatives, and adding 
conditions \ref{item: rs 4} and \ref{item: rs 5} to make the choice unique. Figure~\ref{fig: 
BT} shows the first even edges of $\cT$ labeled with representatives of $\cY$, in the simple case $p=2$.

\begin{figure}
\makeatletter
\tikzset{nomorepostaction/.code={\let\tikz@postactions\pgfutil@empty}}
\makeatother
\begin{tikzpicture}[-,>=triangle 45,shorten >=0pt,auto,node distance=3cm,thick,enode/.style={circle,inner sep = 2pt,fill=blue!30,draw,font=\sffamily\bfseries},onode/.style={circle,inner sep = 3pt,fill=red!70!blue!30,draw,font=\sffamily\bfseries},midar/.style={thick,postaction={nomorepostaction,decorate,decoration={markings,mark=at position 0.5 with {\arrow{>}}}}} ,scale=0.7, every node/.style={scale=1.0}]

  \node[enode,label = 1100:$v_*$] (g0) {$\phantom{\gamma_0}\gamma_0\phantom{\gamma_0}$};
  \node[onode,label = 1-1-1:$\hat v_*$] (g0h) [above of=g0] {$\phantom{a}\tilde \gamma_0\phantom{a}$};
  \node[onode] (g1) [below left of=g0] {$\phantom{a}\gamma_1\phantom{a}$};
  \node[onode] (g2) [below right of=g0] {$\phantom{a}\gamma_2\phantom{a}$};
  \node[enode] (g1h) [above left of=g0h] {$\phantom{a}\tilde\gamma_1\phantom{a}$};
  \node[enode] (g2h) [above right of=g0h] {$\phantom{a}\tilde\gamma_2\phantom{a}$};
  \node[onode] (g1g1h) [left of=g1h] {$\gamma_1\tilde\gamma_1$};
  \node[onode] (g2g1h) [above of=g1h] {$\gamma_2\tilde\gamma_2$};
  \node[onode] (g1g2h) [right of=g2h] {$\gamma_1\tilde\gamma_1$};
  \node[onode] (g2g2h) [above of=g2h] {$\gamma_2\tilde\gamma_2$};
  \node[enode] (g1hg2) [right of=g2] {$\tilde\gamma_1\gamma_2$};
  \node[enode] (g2hg2) [below of=g2] {$\tilde\gamma_2\gamma_2$};
  \node[enode] (g1hg1) [left of=g1] {$\tilde\gamma_1\gamma_1$};
  \node[enode] (g2hg1) [below of=g1] {$\tilde\gamma_2\gamma_1$};
  \node[onode] (g1g2hg1) [below left of=g2hg1] {\tiny$\gamma_1\tilde\gamma_2\gamma_1$};
  \node[onode] (g2g2hg1) [below right of=g2hg1] {\tiny$\gamma_2\tilde\gamma_2\gamma_1$};
  \node[enode] (g1hg2g2h) [above left of=g2g2h] {\tiny$\tilde\gamma_1\gamma_2\tilde\gamma_2$};
  \node[enode] (g2hg2g2h) [above right of=g2g2h] {\tiny$\tilde\gamma_2\gamma_2\tilde\gamma_2$};
  \path[every node/.style={font=\sffamily\bfseries\small}]
      (g0) edge[midar] node [right] {$\gamma_0 = \tilde\gamma_0 = 1$} (g0h)
           edge[midar] node [left] {$e_*$} (g0h)
           edge[midar] node [left] {$\gamma_1$} (g1)
           edge[midar] node [right] {$\gamma_2$}  (g2)
      (g1h) edge[midar] node [right] {$\tilde\gamma_1$} (g0h)
            edge[midar] node [below] {$\gamma_1\tilde\gamma_1$} (g1g1h)
            edge[midar] node [right] {$\gamma_2\tilde\gamma_1$} (g2g1h)
      (g2h) edge[midar] node [left] {$\tilde\gamma_2$} (g0h)
            edge[midar] node [below] {$\gamma_1\tilde\gamma_2$} (g1g2h)
            edge[midar] node [right] {$\gamma_2\tilde\gamma_2$} (g2g2h)
      (g1hg1) edge[midar] node [below] {$\tilde\gamma_1\gamma_1$} (g1)
      (g2hg1) edge[midar] node [right] {$\tilde\gamma_2\gamma_1$} (g1)
      (g1hg2) edge[midar] node [below] {$\tilde\gamma_1\gamma_2$} (g2)
      (g2hg2) edge[midar] node [right] {$\tilde\gamma_2\gamma_2$} (g2)
      (g1hg2g2h) edge[midar] node [left] {$\tilde\gamma_1\gamma_2\tilde\gamma_1$} (g2g2h)
      (g2hg2g2h) edge[midar] node [left] {$\tilde\gamma_2\gamma_2\tilde\gamma_1$} (g2g2h)
      (g2hg1)  edge[midar] node [left] {$\gamma_1\tilde\gamma_2\gamma_1$} (g1g2hg1)
      (g2hg1)  edge[midar] node [right] {$\gamma_2\tilde\gamma_2\gamma_1$} (g2g2hg1);
\end{tikzpicture}

\caption{Vertices and edges of the Bruhat-Tits tree labeled using the radial system ($p=2$). 
Blue vertices are even and
red ones are odd, and only the even edges are shown.}\label{fig: BT}
\end{figure}
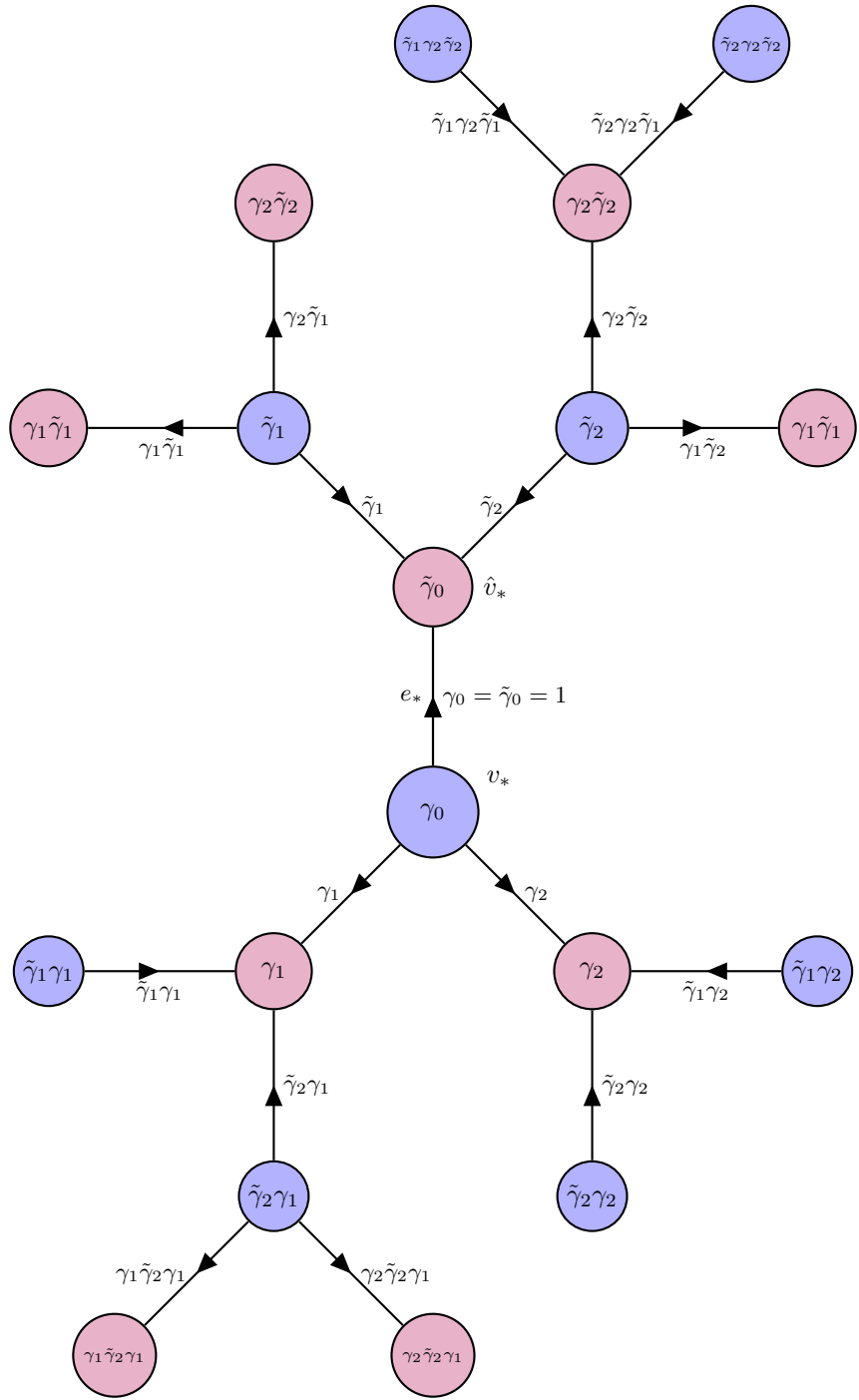

\subsection{Measures on $\P^1(\Q_p)$} 

Let $\cB(\PP^1(\Q_p))$ be the set of compact-open balls in $\PP^1(\Q_p)$, which forms a basis for the topology of
$\PP^1(\Q_p)$. There is a $\GL_2(\QQ_p)$-equivariant bijection \[ \begin{array}{ccc}\cE &\stackrel{\cong}{\lra} 
& \cB(\PP^1(\QQ_p)) \\ e & \longmapsto & U_e \end{array}\] sending $e_*$ to $\ZZ_p$. Therefore, if $\gamma(e)=e_*$ then
$U_e = \gamma^{-1}\ZZ_p$; in particular $U_e=\gamma_e^{-1} \Z_p$. Under this bijection an open ball $U_{e}$ is contained
in $U_{e'}$ if and only if there is a path (directed and without backtracking) in $\cT$ having initial edge $e$ and
final edge $e'$.

The following basic lemma will be useful in Section \ref{sec: overconvergent}. We denote by $|U|$ the
diameter of an open ball $U\in\cB(\PP^1(\QQ_p))$.
\begin{lemma}
 Let $g_i=s_i^{-1} = \omega_p\gamma_i$. For each $r\geq 0$ denote by $\cB(\ZZ_p,p^{-r})$ the set of open balls $U\subseteq \ZZ_p$ of
diameter $p^{-r}$. Then for all $r\geq 0$:
\begin{align}\label{eq: balls of radius p**r}
\cB(\ZZ_p,p^{-r}) = \{(g_{i_1}\cdots g_{i_r})^{-1}\ZZ_p~|~  1\leq i_k\leq p\}.
\end{align}
\end{lemma}
\begin{proof}
We do induction on $r$, and note that the case of $r=0$ is trivial since both sets consist of only one open, 
namely $\ZZ_p$.

Note that $g_i^{-1}\ZZ_p\subset \ZZ_p$, and actually
\[
\ZZ_p=\coprod_{i=1}^p g_i^{-1}\ZZ_p.
\]
This follows from the local form of the $g_i$ as in \eqref{eq: local s_i}: for $i\geq $1,
$\iota_p(g_i^{-1})=\smtx{p}{-i}{0}{1}u_i$ with the $u_i\in\Gamma_0^{\text{loc}}(p)$.
Therefore, we obtain the inclusion $\supseteq$ in \eqref{eq: balls of radius p**r}. The set $\cB(\ZZ_p,p^{-r})$ has size
$p^r$, so it only remains to show that:
\[
(i_1,\ldots,i_r)\neq (j_1,\ldots,j_r)\implies (g_{i_1}\cdots g_{i_r})^{-1}\ZZ_p \neq (g_{j_1}\cdots g_{j_r})^{-1}\ZZ_p.
\]
Again, the previous decomposition of $\ZZ_p$ and the induction hypothesis prove the above claim.
\end{proof}
From the above lemma we deduce:
\begin{corollary}\label{cor: opens in Z_p}
\begin{enumerate}
\item An open ball $U_e$ corresponding to an even edge $e$ is contained in $\Z_p$ if and only if  $\gamma_e$ is of the form
\[
\gamma_e=\tilde\gamma_{i_1}\gamma_{j_1}\cdots
\tilde\gamma_{i_n}\gamma_{j_n} \text{, with all $i_k,j_k\in\{1,\ldots,p\}$ and some $n\geq 0$.}
\]
\item An open ball $U_{\ol e}$ corresponding to the opposite of an even edge $e$ is contained in $\Z_p$ if and only if $\gamma_e$ is of the form
\[
\gamma_e=\gamma_{j_1}\tilde\gamma_{i_2}\gamma_{j_2}\cdots
\tilde\gamma_{i_n}\gamma_{j_n}
\]
with all $i_k,j_k\in\{1,\ldots,p\}$, and some $n\geq 0$.
\end{enumerate}
\end{corollary}
\begin{proof}
  Note that $\tilde\gamma_{i_k}\gamma_{j_k} = p^{-1}g_{i_k}g_{j_k}$, so
\[
(\tilde\gamma_{i_k}\gamma_{j_k})^{-1}\ZZ_p = (g_{i_k}g_{j_k})^{-1}\ZZ_p.
\]
Now the first claim follows from the fact that even edges correspond to balls of diameter $p^{-2n}$ for some $n\geq 0$
and
the lemma. The second claim is similar.
\end{proof}

Let $\meas_0(\P^1(\Q_p),\Z)$ denote the set of
$\Z$-valued measures on $\P^1(\Q_p)$ of total measure $0$. It acquires the structure of left
$\GL_2(\Q_p)$-module as follows: for $m\in \meas_0(\P^1(\Q_p),\Z)$ and ${g\in\GL_2(\Q_p)}$
\[
 (g m)(U)=m(g^{-1}U) \ \text{ for all compact-open } U.
\]

Let $\cF(\cE,\Z)$ denote the set of functions from $\cE$ to $\Z$ and let
\[
 \cF_0(\cE,\Z)=\{ c\in\cF(\cE,\Z) \colon c(e)=-c(\bar e) \text{ for all } e\in\cE\}.
\]

A \emph{$\Z$-valued harmonic cocycle} is a function $c\in\cF_0(\cE,\Z)$ such that
\[
 \sum_{s(e)=v}c(e)=0\ \text{ for all } v\in \cV.
\]
The bijection $\cE\leftrightarrow \cB(\PP^1(\Q_p))$ induces an identification between $\meas_0(\P^1(\Q_p),\Z)$ and $\hc(\Z)$.

\begin{remark}
  The module $\cF_{\text{har}}(\Z)$ also appears in the theory of modular forms. Indeed, the Jacquet--Langlands correspondence and the theory of Cerednik--Drinfeld relate harmonic cocycles that are invariant with respect to arithmetic subgroups of \emph{definite} quaternion algebras of discriminant $D$ to $pD$-new modular forms (see, e.g., \cite[\S5]{darmon-book}). However, in this work we only consider \emph{indefinite} quaternion algebras. In this case, the corresponding invariant harmonic cocycles are trivial, and one needs to look at higher cohomology groups (cf. \S\ref{subsection: Greenberg construction} below), hence deviating from the more classical theory.
\end{remark}

\section{Quaternionic $p$-adic Darmon points}\label{subsection: Greenberg construction}
This section is devoted to reviewing Greenberg's construction of quaternionic $p$-adic Darmon points \cite{Gr} in the case
of
elliptic curves over $\Q$. Recall that in this setting $E$ is an elliptic curve over $\Q$ of
conductor $N=pDM$, and $K$ a real quadratic field in which all primes dividing $pD$ are inert and all primes dividing
$M$ are
split.

The method attaches to any embedding of $\Z[1/p]$-algebras $\psi: \cO_K\hookrightarrow R$ a Darmon point
$P_\psi\in E(K_p)$, which is the image under Tate's uniformization map of a certain
quantity $J_\psi\in K_p^\times$. The construction of $J_\psi$ can be divided into three stages: 
\begin{enumerate}
 \item Construct a $1$-cohomology 
class $[\tilde\mu]=[\tilde\mu_E]\in H^1(\Gamma,\meas_0(\P^1(\Q_p),\Z))$ associated to 
$E$;
\item Construct a $1$-homology class
$[c_\psi]\in H_1(\Gamma,\Div^0(\cH_p))$ associated to $\psi$; and
\item Set $J_\psi=\Xint\times \langle
[c_\psi],[\tilde\mu] \rangle$, where $\Xint\times\langle\ ,\ \rangle$ is a certain  ``integration pairing'' .
\end{enumerate}
We describe each one of the steps separately.

\subsection{The cohomology class attached to $E$}\label{subsection: the cohomology class} Recall the two characters
$\lambda_E^\pm$ of the Hecke algebra associated to $E$ in \eqref{eq: def of lambda}. Choose a sign $\sigma\in\{\pm\}$ 
and consider the character $\lambda = \lambda_E^\sigma$. If we denote by $H^1(\GMp,\Z)_{p\text{-new}}$ the $p$-new 
subspace
(see, e.g. \cite[\S3]{Gr} for the definition), then the submodule $\left(H^1(\GMp,\Z)_{p\text{-new}}\right)^\lambda$  is
free of rank $1$. In fact, the coboundary group
$B^1(\GMp,\Z)$ is
trivial (for $\GMp$ acts trivially on $\Z$), so there exists a cocycle $\varphi=\varphi_E\in 
Z^1(\GMp,\Z)_{p\text{-new}}$ such that: 
\begin{enumerate}
 \item $\opT_\ell  \varphi=a_\ell \varphi$ for all primes $\ell\nmid pMD$, \item $\opU_\ell  \varphi=a_\ell \varphi$
for all $\ell\mid pM$,\item $\opW_\infty
\varphi=\sigma \varphi$, and
 \item the image of $\varphi$ is not contained in any proper ideal of $\Z$.
\end{enumerate}
The cocycle $\varphi$ is uniquely determined, up to sign, by these conditions and therefore me may and do fix such a cocycle $\varphi$. The following theorem can be seen as a generalization
of \cite[Proposition 1.3]{darmon-pollack} to the case where $B$ is a
division algebra.
 \begin{theorem}[Greenberg \cite{Gr}]\label{th: greenberg} There exists $\tilde \mu\in Z^1(\Gamma,\meas_0(\P^1(\Q_p,\Z)))$ whose cohomology class $[\tilde\mu]\in H^1(\Gamma,\meas_0(\P^1(\Q_p),\Z))$ 
satisfies:
\begin{enumerate}
 \item $\opT_\ell  [\tilde\mu]=a_\ell [\tilde\mu]$ for all primes $\ell\nmid pM$; \item $\opU_\ell [\tilde\mu]=a_\ell 
[\tilde\mu]$
for all $\ell\mid M$, \item $\opW_\infty
 [\tilde\mu]=\sigma [\tilde\mu]$, and
 \item\label{eq: property of m} $\tilde\mu_{\gamma}(\Z_p)=\varphi_\gamma$ for all $\gamma\in \GMp$.
\end{enumerate}
In addition,  $[\tilde\mu]$ is uniquely determined by this conditions.
\end{theorem}

One can think of the cocycle $\tilde\mu$ as a ``system of measures'': for any
$\gamma\in\Gamma$ there is an associated measure $\tilde\mu_\gamma$. A cocycle $\tilde\mu$ as in the above theorem can be explicitly constructed by
applying the methods of \cite[\S4.2]{LRV} as follows. First of all we need to define a related cocycle $\mu=\mu_E$, 
which
 will actually play an important role in our explicit algorithms. Given $e\in\cE^+$ and $g\in \Gamma$ let
$h(g,e)$ be the 
element
of $\GMp$ defined by the equation 
\begin{align}\label{eq: definition of h}
\gamma_eg=h(g,e)\gamma_{g^{-1}(e)}.
\end{align}
Recall $\cY=\{\gamma_e\}_{e\in\cE^+}$ the radial system fixed in Definition \ref{def: radial system}. For
$g\in\Gamma$ let $\mu_g\in\cF(\cE_0,\Z)$ be the function defined by
\begin{align}\label{eq: definition of mu}
 \mu_g(e)=\varphi_{h(g,e)}, \ \text{if $e\in \cE^+$}.
\end{align}
This condition already determines the values of $\mu_g(e)$ for $e\in\cE^-$, for if $\mu_g$ belongs 
to $\cF_0(\cE,\Z)$ then $\mu_g(e)=-\mu_g(\bar e)$ and $\bar e \in\cE^+$. The map $g\mapsto \mu_g$
defined this way turns out to be a $1$-cocycle. 

Fix  a prime $r$ not dividing $N$ and set $t_r=(T_r-r-1)\in\mathbb{T}^{(pD)}$. The following
proposition, which essentially restates results of \cite{Gr} and \cite{LRV}, claims that $[\tilde \mu]$ can be
computed from $t_r[\mu]$. 
\begin{proposition}\label{prop: proposition about m and mu}
 The cocycle $\mu$ belongs to $Z^1(\Gamma,\meas_0(\PP^1(\Q_p),\Z))$, and $t_r[\mu]$ is a multiple of the cohomology class $[\tilde\mu]$ 
given by 
Theorem \ref{th: greenberg}.
\end{proposition}
\begin{proof} Recall the identification $\meas_0(\P^1(\Q_p),\Z)$ with $\hc(\Z)$. First of all, since $\varphi$ belongs to $\left(H^1(\GMp,\Z)_{p\text{-new}}\right)^\lambda$ by Remark
\ref{rk: its shapiro} and the fact that the isomorphism of Shapiro's Lemma commutes with the Hecke action \cite[Lemma
1.1.4]{ash-stevens} we see that $[\mu]\in H^1(\Gamma,\cF_0(\cE,\Z))^\lambda$.
 The system $\cY$ used to define $\mu$ is radial, and by
 \cite[Proposition 4.8]{LRV} this implies that $\mu_g$ belongs to $\hc(\Z)$ for all $g\in\Gamma$. In particular
$[\mu]$ can be viewed as an element of $H^1(\Gamma,\hc(\Z))$.  The natural map 
\[
 \rho:  \Q\otimes H^1(\Gamma,\hc(\Z))\lra  \Q\otimes H^1(\Gamma,\cF_0(\cE,\Z))_{p\text{-new}}
\]
is surjective but not injective: its  kernel is $H^1(\Gamma,\hc(\Z))^{\deg}$ (see \cite[\S8]{Gr}). Since $\lambda$ arises from a cuspidal eigenform,
$\lambda(\opT_r)$ is not $r+1=\deg(T_r)$ and thus $\opT_r-r-1$ projects to the complementary of 
${\Q\otimes H^1(\Gamma,\hc(\Z))^{\deg}}$, and that
it acts as multiplication by $a_r-r-1$ on $\Q\otimes H^1(\Gamma,\hc(\Z))_{p\text{-new}}^\lambda$.
\end{proof}
In view of this result there exists an integer $c_r$ such that $t_r[\mu]=c_r[\tilde\mu]$. We abuse the notation to
denote $c_r^{-1}t_r$ simply as $t_r$, so that we have an equality $[\tilde\mu]=t_r[\mu]$.

\begin{remark}\label{rk: its shapiro}
 In fact, the cohomology class of $[\tilde\mu] \in H^1(\Gamma,\cF_0(\cE,\Z))$ is nothing but the image of $\varphi$ 
under the
isomorphisms \[H^1(\GMp,\Z)\simeq
H^1(\Gamma,\operatorname{coind}_{\GMp}^\Gamma(\Z))\simeq H^1(\Gamma,\cF_0(\cE,\Z)),\] where the first isomorphism is
given by Shapiro's Lemma and the second comes from the isomorphism $\operatorname{coind}_{\GMp}^\Gamma(\Z)\simeq
\cF_0(\cE,\Z)$ (cf. \cite[Corollary 16]{Gr}).
\end{remark}

\subsection{The homology class attached to $\psi$}\label{subsection: the homology class} Let $\cH_p=K_p\setminus
\Q_p$ be the $K_p$-rational points of the $p$-adic upper half plane. The group $\psi(\cO_K^\times)$ acts
on $\cH_p$ via the isomorphism $\iota_p:B\otimes\Q_p\simeq
\M_2(\Q_p)$. Since $p$ is inert in $K$ the action has two fixed points; let
$\tau_\psi\in \cH_p$ be one of them. Let also $\varepsilon_K\in \cO_K^\times$ be a unit of norm $1$, and set
$\gamma_\psi=\psi(\varepsilon_K)$. Since $\gamma_\psi\tau_\psi=\tau_\psi$, the element $\gamma_\psi\otimes \tau_\psi$ belongs to  $Z_1(\Gamma,\Div \cH_p)$. From
the exact sequence 
\begin{align}\label{eq: short exact sequence}
 0\lra \Div^0\cH_p\lra \Div \cH_p \stackrel{\deg}{\lra} \Z \lra 0
\end{align}
we obtain the long exact sequence in $\Gamma$-homology
\begin{align}\label{eq: exact sequenc homology}
 \cdots \lra H_2(\Gamma,\Z)\stackrel{\delta}{\lra} H_1(\Gamma,\Div^0\cH_p)\lra H_1(\Gamma,\Div 
\cH_p)\stackrel{\deg_*}{\lra} 
H_1(\Gamma,\Z)\lra \cdots
\end{align}
where $\delta$ is the connecting homomorphism. The group $H_1(\Gamma,\Z)$ is isomorphic to the abelianization  of
$\Gamma$, which is finite (see, e.g., \cite[\S 
2]{LRV-special-values}). If we let $e_\Gamma$ denote its exponent, then
$e_\Gamma[\gamma_\psi\otimes \tau_\psi]$ has a preimage  $[c_\psi]\in H_1(\Gamma,\Div^0\cH_p)$, and this is the  
homology
class
attached to $\psi$ we were looking for.
\begin{remark}\label{rk: c_psi is well defined up to}
 The homology class $[c_\psi]$ is well-defined up to elements in $\delta\left( H_2(\Gamma,\Z) \right)$.
\end{remark}

\subsection{Integration pairing and Darmon points}
Let $f:\P^1(\Q_p)\ra K_p^\times$ be a continuous function and let  $m\in\meas_0(\P^1(\Q_p),\Z)$. The 
\emph{multiplicative 
integral} of $f$ with respect to $m$ is defined as the limit of Riemann products
\[
\Xint\times_{\P^1(\Q_p)}f(t)d m(t)=  \lim_{||\cU||\ra 0}\prod_{U\in\cU}f(t_U)^{m(U)}\in K_p^\times,
\]
where the limit is taken over increasingly finer finite coverings $\cU$ of $\P^1(\Q_p)$ by compact-opens, and $t_U$ is 
any sample point in $U$. If $U\subset\P^1(\Q_p)$ it is customary to denote
\[
 \Xint\times_Uf(t)dm(t)=\Xint\times_{\P^1(\Q_p)}f(t)\charf_U(t)dm(t).
\]
For 
$D\in\Div^0(\cH_p)$ let $f_D:\P^1(\Q_p)\ra K_p^\times$ be a function with divisor $D$ (for instance, if
$D=(\tau_0)-(\tau_1)$ 
one can take $f_D(t)=\frac{t-\tau_0}{t-\tau_1}$). Observe that $f_D$ is well-defined up to multiplication by scalars in 
$K_p^\times$; nevertheless, since these scalars integrate to $1$ there is a well defined pairing
\[
 \begin{array}{ccc}\Div^0(\cH_p)\times \meas_0(\P^1(\Q_p),\Z) & \lra &
K_p^\times\\ 
  (D,m) & \longmapsto &\displaystyle \Xint\times_{\P^1(\Q_p)}f_D(t)dm(t).
 \end{array}
\]
By cup product this defines a pairing
\[
\xymatrix@C15pt@R1pt{
H_1(\Gamma,\Div^0(\cH_p))\times H^1(\Gamma ,\meas_0(\P^1(\Q_p),\Z))\ar[r]^-{\displaystyle\Xint\times\langle \ , \ \rangle}&K_p^\times\\
(\sum_g g\tns D_g,\xi)\ar@{|->}[r] & \displaystyle\prod_g\Xint\times_{\P^1(\Q_p)}f_{D_g}(t)d\xi_g(t),
}
\]
which is equivariant for the Hecke action: 
\begin{align}\label{eq: integration pairing is equivariant}
\Xint\times \langle T_\ell \sum_g g\tns D_g,\xi \rangle=\Xint\times \langle \sum_g g\tns D_g,T_\ell\xi \rangle . 
\end{align}
Define \[L=\left\{\displaystyle\Xint\times\langle
\delta c, [\tilde\mu]\rangle \colon c\in  
H_2(\Gamma,\Z) \right\}\subset K_p^\times,\] where $[\tilde\mu]=t_r[\mu]$ is the cohomology class associated
to $E$
in Section \ref{subsection: the cohomology class}. It turns out that $L$ is a lattice in $K_p^\times$
\cite[Proposition 30]{Gr}. The following key result, which was independently proven by Dasgupta--Greenberg and 
Longo--Rotger--Vigni, relates $L$ to the Tate lattice of $E$.
\begin{theorem}[\cite{greenberg-dasgupta},\cite{LRV}]
 The lattice $L$ is commensurable to the Tate lattice $\langle q_E\rangle$ of $E/K_p$.
\end{theorem}
Thanks to this theorem one can find an isogeny  $\beta: K_p^\times/L\ra K_p^\times/\langle q_E\rangle$. Denote by 
$\Phi_{\text{Tate}}:K_p^\times/\langle q_E\rangle\ra E(K_p)$ Tate's uniformization map and let
\[J_\psi=\Xint\times\langle c_\psi,[\tilde\mu]\rangle.\] Observe that $J_\psi$ is a well-defined quantity in
$K^\times/L$ thanks to Remark
\ref{rk: c_psi 
is well defined up to}.

\begin{conjecture}[Greenberg]
The local point $P_\psi=(\Phi_{\text{Tate}}\circ \beta) (J_\psi)\in E(K_p)$ is a global point. More
precisely, it is rational 
over the narrow Hilbert class field $H_K^+$ of $K$.
\end{conjecture}
\begin{remark}\label{rk: integration pairing is equivariant}
The integration pairing is equivariant with respect to the Hecke action, so $J_\psi$ can also be computed as
\begin{align}\label{eq: the J_psi we compute}
 J_\psi=\Xint\times\langle [t_rc_\psi],[\mu]\rangle.
\end{align}
\end{remark}

\section{The effective computation of quaternionic $p$-adic Darmon points}\label{sec: the algorithm}
In this section we present the explicit algorithms that allow for the effective calculation of the quaternionic $p$-adic Darmon points. As we reviewed in \S\ref{subsection: Greenberg construction}, this amounts to compute the cohomology class associated to the elliptic curve, the homology class corresponding to an optimal embedding, and the integration pairing. 

In \S\ref{subsection: Computation of the cohomology class} we show how to compute the cohomology  class (the main algorithmic result is given in Theorem \ref{th: reduce in amalgam}), and in  \S\ref{subsection: Computation of the homology class} how to compute the homology class (the main algorithm is stated as Theorem \ref{th:homology-general}). In fact, these two algorithms are already enough to compute the Darmon points, as one can then evaluate the integration pairing via Riemann products, which can be thought of as the most naive method of integration. This is briefly recalled in \S\ref{subsection: computation riemann}.
 
Finally, in \S\ref{subsection: example riemann} we illustrate the use of this method by giving a detailed explicit example of a Darmon point calculated with the algorithms introduced this section, together with Riemann products for approximating the integrals. This also serves as a motivation for Section \ref{sec: overconvergent}, because even though in principle is possible to compute the integrals using Riemann products, it is too computationally costly. Section \ref{sec: overconvergent} will be devoted to an efficient method for calculating the type of integrals arising in $p$-adic Darmon points.

\subsection{Computation of the cohomology class} \label{subsection: Computation of the cohomology class}
The first step is to calculate a cocycle $\varphi\in 
H^1(\GMp,\Z)_{p\text{-new}}$ that lies in the $\lambda$-isotypical component by the Hecke action. We remark that there are algorithms for effectively dealing with arithmetic subgroups of indefinite quaternion division algebras. More concretely, there are algorithms that:
\begin{itemize}
 \item compute a presentation of $\GM$ and $\GMp$ in terms of generators and relations, and
 \item express an element of $\GM$ or $\GMp$ as a word in the generators.
\end{itemize}
These algorithms were introduced by John Voight \cite{voight-fundamental} and are implemented in Magma \cite{magma}.

 Note that we have
\[
 H^1(\GMp,\Z)=\Hom(\GMp,\Z)=\Hom(\GMp_{\text{ab}},\Z),
\]
and that the finitely generated abelian group $\GMp_\text{ab}$ is easy to calculate from an explicit presentation of
$\GMp$. Using this description and formula \eqref{eq: def of T_ell} one can algorithmically compute the Hecke action on
$H^1(\GMp,\Z)$ (cf. \cite{greenberg-voight} for more details). 

Using the Atkin--Lehner operator $W_p$ one computes the $p\text{-new}$ part of the group $H^1(\GMp,\Z)$, and one then proceeds to diagonalize it with respect to several Hecke operators
$T_\ell$, until the common eigenspace corresponding to $\lambda$ has rank $2$. In practice, a few values of $\ell$ are usually 
enough. Then the space where the Atkin--Lehner operator $W_\infty$ acts with sign $\sigma\in\{\pm 1 \}$ has rank
$1$, and we can take $\varphi$ to be one of its generators. 

The final step is to compute the values of $\mu$ by means of formula \eqref{eq: definition of mu}. In order to do so, one needs to be  able to express any element  
$g\in \Gamma$ as $g=h(g)\gamma_e$, where $h(g)\in\GMp$ and $\gamma_e\in\cY$. In the next theorem we show that this can be, indeed, computed in an algorithmic fashion. 

\begin{theorem}\label{th: reduce in amalgam}
  There is an algorithm that, given $g\in\Gamma$, outputs $h(g)\in\GMp$ and $\gamma_e\in\cY$ such that 
$g=h(g)\gamma_e$, in time proportional to the distance from $e_*$ to $e=g^{-1}(e_*)$
\end{theorem}
In order to describe the algorithm and prove its correctness, it is useful to recall the notion of
\emph{distance between lattices} (cf. \cite[Chapter II,
\S 1.1]{serre-trees}). If $\Lambda$ and $\Lambda'$ are lattices in $\Q_p^2$ there exists a basis $\{b_1,b_2\}$ for 
$\Lambda$ such that $\{p^{x}b_1, p^{y}b_2\}$ is a basis for $\Lambda'$ for certain $x,y\in\Z$. Then the distance
$d(\Lambda,\Lambda')$ is defined to be $|x-y|$. It is independent of the choice for $\{b_1,b_2\}$, and it only depends
on the homothety classes of $\Lambda$ and $\Lambda'$. In addition, this notion of distance coincides with the distance
in the Bruhat--Tits tree; that is to say, if $\Lambda$ and $\Lambda'$ represent vertices $v$ and $v'$ in $\cV$ then
$d(\Lambda,\Lambda')=d(v,v')$.

Under the correspondence $\GMp\setminus\Gamma\leftrightarrow \cE^+$ an element $g\in \Gamma$ is associated with 
$e=g^{-1}(e_*)\in\cE^+$. Its source $s(e)=g^{-1}(v_*)$ is then represented by the lattice $g^{-1}(\Z_p\oplus
\Z_p)$,  and its target $ t(e)=g^{-1}(\hat v_*)$  by the lattice $g^{-1}(\Z_p\oplus p\Z_p)$. Thus, if we let
$\iota_p(g^{-1})=\smtx a b c d$, the
columns $\smtx a b c d$ are a basis for the lattice
$s(e)$ and the columns of $\smtx{a}{bp}{c}{dp}$ are a basis for $t(e)$. The distances $d(s(e),v_*)$ and $d(t(e),v_*)$
are easily read from the Smith normal form of these matrices: 
if
\[
 \smtx a b c d = G \smtx{d_1}{0}{0}{d_2}H,\ \ \smtx{a}{pb}{ c}{ pd} = G' \smtx{d_1'}{0}{0}{d_2'}H', \text{ for
some } G,G',H,H'\in\GL_2(\Z_p),
\]
then 
\begin{align}\label{eq: distance of lattices}
d(s(e),v_*)=|v_p(d_1)-v_p(d_2)|\ \ \text{ and }\ \ d(t(e),v_*)=|v_p(d_1')-v_p(d_2')|. 
\end{align}
We may identify $g$ with its associated edge $e=g^{-1}(e_*)$, and use expressions
such as $d(s(g),v_*)$ or $d(t(g),v_*)$. We say that $e$ (or $g$) is an \emph{outward} edge if
$d(s(g),v_*)<d(t(g),v_*)$ and that it is  \emph{inward}
otherwise. Observe that one can easily determine whether $g$ is inward or outward by means of formula $\eqref{eq:
distance of lattices}$.

\subsection*{Proof of Theorem \ref{th: reduce in amalgam}} Given an element $g\in\Gamma$ let $e$ be the edge
$g^{-1}(e_*)$. It is enough to compute the representative $\gamma_e\in\cY$, since then
$h(g)=g\gamma_e^{-1}\in\GMp$.

Observe that if $g$ is outward and $d(s(g),v_*)=0$ then $\gamma_e$ equals some
$\gamma_i\in\Upsilon$ (see the edges leaving $v_*$ in Figure \ref{fig: BT}), and it is easily computed since it is the
single $\gamma_i$ such that $\gamma_i^{-1}g\in \GMp$. For general $g$ the algorithm consists on recursively reducing to
this particular case as follows: 
\begin{enumerate}
 \item\label{enumerate: first step in the algorithm} If $g$ is outward and $d(s(g),v_*)>0$, then there exists a single 
$\gamma_i$ such
that $\gamma_i^{-1}g$ is associated with an inward edge. Compute such $\gamma_i$ and set
$g=\gamma_i^{-1}g$.
\item\label{enumerate: second step in the algorithm} If $g$ is inward, then there exists a single
$\tilde\gamma_i$ such that $\tilde\gamma_i^{-1}g$ is outward. In addition, for such $\tilde\gamma_i$ we have that
$d(s(\tilde\gamma_i^{-1}g),v_*)<d(s(g),v_*)$. Set $g=\tilde\gamma_i^{-1}g$.
\item If $g$ is outward and $d(s(g),v_*)=0$ compute the single $\gamma_i$ such that $\gamma_i^{-1}g\in\GMp$ and end the
algorithm. Otherwise go to step \ref{enumerate: first step in the algorithm}.
\end{enumerate}
Every time we run step \ref{enumerate: second step in the algorithm} the distance $d(s(g),v_*)$ decreases, so the
algorithm terminates. The representative $\gamma_e$ is then the product of all the $\gamma_i$ and $\tilde
\gamma_j$ computed in each step. Finally, it is clear that the number of stages is  $d(s(e),v_*)$.

\qed

\subsection{Computation of the homology class}\label{subsection: Computation of the homology class} Given the real
quadratic field $K$ and its ring of integers 
$\cO_K=\Z[\omega]$, the first step is to compute an embedding of $\Z[1/p]$-algebras $\cO_K\hookrightarrow R$. 
In fact, thanks to our running assumptions on $K$ we can find $\Z$-algebra embeddings 
$\cO_K\hookrightarrow R_0(M)$. Computing them in practice amounts to finding elements in $B$ whose reduced 
norm and trace coincide with that of $\omega$, and one can use the routines of Magma \cite{magma} to compute them (e.g. 
the routine  \texttt{Embed( , )}). 

Every embedding $\psi_0:\cO_K\hookrightarrow R_0(M)$ induces $\psi\colon\cO_K\hookrightarrow R$ via the inclusion $R_0(M)\subset R$,
giving rise to the $1$-cycle $\gamma_\psi\otimes \tau_\psi$ in
$Z^1(\Gamma,\Div\cH_p)$ via the process described in Section \ref{subsection: the homology class}.
Denote by $e_\Gamma$ the exponent of $H_1(\Gamma,\Z)$, so 
that the element
\[
e_\Gamma[\gamma_\psi\otimes \tau_\psi]\in H_1(\Gamma,\Div\cH_p)
\]
lifts under $\deg_*$ to an element $[c_\psi]\in
H_1(\Gamma,\Div^0\cH_p)$
(cf. the exact sequence \eqref{eq: exact sequenc homology}).

We devote the rest of this subsection to describe an algorithm for computing $c_\psi$.  Note that once $c_\psi$  is found,  it is easy to compute $\tilde c_\psi$ by means of formula~\eqref{eq: def of T_ell}.

Let $\langle X \mid R\rangle$ be a presentation of $\Gamma$, where $X=\{x_1,\ldots,x_n\}$ are the generators and
$R=\{r_1,\ldots,r_m\}$ the relations. It can be explicitly computed by means of Voight's algorithms, which provide
presentations for $\GM$ and $\GMp$, say \[\GM=\langle Y \mid S \rangle\ \text{ and }\ \GMp=\langle Z \mid
 T\rangle.\] A set of generators of $\GMhat$ is $\hat Y=\{\hat y_i:=\omega_p^{-1}y_i\omega_p\colon y_i\in Y\}$, and a 
set of relations $\hat S$ is that in which the $\hat y_i$ satisfy the same relations as the $y_i$. Then each $z\in Z$ 
can be expressed as a word in the generators of $Y$, that we denote $\alpha(z)$, and as a word in the generators of 
$\hat Y$, that we denote $\hat\alpha(z)$. If we let
$S_Z=\{\alpha(z)\hat\alpha(z)^{-1}\colon z\in Z\}$ then a
presentation of $\Gamma=\GM \star_{\GMp}\GMhat$ is given by
\[
 \langle X\mid R\rangle =\langle Y\cup \hat Y\mid S\cup \hat S\cup S_Z \rangle. 
\]
Any $g\in \GM$ can be expressed as a word in $Y$ by means of Voight's algorithm~\cite{voight-fundamental}. Combining
this with the algorithm of Theorem \ref{th: reduce in amalgam} we obtain an algorithm for expressing any $g\in \Gamma$
as a word in $X$. 

The following notation will be useful in describing the algorithm for computing $c_\psi$: If $w$ is a word and
$x\in X$, we define $v_x(w)\in\ZZ$ as the sum of the exponents of $x$ appearing in $w$. We also set $v_X(w) =
(v_{x_1}(w),\ldots,v_{x_n}(w))$. For example, if $w = x_1^3x_2^3x_3^{-1}x_1^{-2}x_3^3$, then $v_X(w) = (1,3,2)$.

The first step in lifting $e_\Gamma[\gamma_\psi\otimes\tau_\psi]=[\gamma_\psi^{e_\Gamma}\otimes\tau_\psi]$ consists in computing
$e_\Gamma$. This is easily obtained using integral linear algebra to obtain the structure of $\Gamma_{\text{ab}}$ from 
the presentation of $\Gamma$.

Next, one obtains a word representation $w$ for $\gamma_\psi^{e_\Gamma}$. Since we are assuming that
$\gamma_\psi^{e_\Gamma}$ is trivial in $H_1(\Gamma,\ZZ)\cong \Gamma_{\text{ab}}$, the vector  $v_X(w)$ belongs to
the image of the abelianized relations, say $v_X(w) = a_1 v_X(r_1) + \cdots +a_k v_X(r_m)$. We consider instead the word
$w'=w r_1^{-a_1}\cdots r_m^{-a_m}$, which represents the same element $\gamma_\psi^{e_\Gamma}\in\Gamma$, but which satisfies
$v_X(w') = 0$.

In what follows we write $\equiv$ to mean equality up to boundaries. The algorithm of Theorem \ref{th:homology-general} below provides a way to find elements $x_i\in\Gamma$ and $D_i\in\Div^0(\cH_p)$ such that
\[
w'\tns \tau_\psi\equiv \sum_{i=1}^n x_i\tns D_i,\ \text{ with the } D_i\in\Div^0(\cH_p),
\]
and therefore to compute $c_\psi=\sum_{i=1}^n x_i\otimes D_i$.

\begin{theorem}
\label{th:homology-general}
 There exists an algorithm that, given $g\in\Gamma$ represented by a word $w$ and given $D\in \Div\cH_p$, computes elements $x_i\in\Gamma$ and $D_i\in\Div^0\cH_p$ such that
\[
g\tns D\equiv \sum_{i=1}^n x_i\tns D_i, \text{ with $\deg(D_i) = v_{x_i}(w)\deg(D)$.}
\]
\end{theorem}
The proof of this theorem consists on making systematic use of the following Lemma.
\begin{lemma} The following relations hold true in $Z_1(\Gamma,\Div\cH_p)$.
\label{lemma:homology-technical}
  \begin{enumerate}
  \item\label{lemma:homology-technical-1} $gh\tns D\equiv g\tns D + h\tns g^{-1}D$.
  \item\label{lemma:homology-technical-2} For all $k\geq 0$, $g^k\tns D \equiv g\tns D'$, with $D' = D+g^{-1}D+\cdots+g^{1-k}D$.
    \item\label{lemma:homology-technical-2bis} $g^{-1}\tns D \equiv -g\tns gD$.
    \item\label{lemma:homology-technical-3} If $gD = D$, then $g^k\tns D \equiv k g\tns D$ for all $k\in \ZZ$.
  \end{enumerate}
\end{lemma}
\begin{proof}
  The first statement is direct from the relation in homology (cf. \eqref{eq: boundary maps}). Note that $D'=g^{-1}D$ has the same degree as $D$.

Next, observe that:
\[
0 \equiv g^{-1}g\tns D \equiv g^{-1}\tns D + g\tns gD,
\]
so we obtain $g^{-1}\tns D\equiv g\tns D'$, with $D'=-gD$ satisfying $\deg(D')=-\deg(D)$, which is the third statement. 
The second statement is proven using induction on $k$, and the last statement is a particular case of the second and 
third ones.
\end{proof}

\subsubsection*{Proof of Theorem~\ref{th:homology-general}}
Suppose that $w=x_{i_1}^{a_1}\cdots x_{i_t}^{a_t}$ is a word representing $g$. Repeated applications of 
Lemma~\ref{lemma:homology-technical}, part \ref{lemma:homology-technical-1} allow to express:
\[
g\tns D \equiv \sum_{s=1}^t x_{i_s}^{a_s}\tns D'_s,\quad \deg(D'_s) = \deg(D).
\]
Using Lemma~\ref{lemma:homology-technical}, part \ref{lemma:homology-technical-2} the above can be rewritten as
\[
g\tns D \equiv \sum_{s=1}^t x_{i_s}\tns D''_s,\quad \deg{D''_s} = a_s\deg(D).
\]
Finally, one can collect the terms involving each of the generators $x\in X$, to obtain:
\[
g\tns D \equiv \sum_{i=1}^n x_i\tns D_i,
\]
and note that $\deg(D_i)= v_{x_i}(w)\deg(D)$, as wanted.

\subsection{Computation of the integration pairing via Riemann products}\label{subsection: computation riemann} In \S\ref{subsection: Computation of the
cohomology class}
and \S\ref{subsection: Computation of the homology class} we have seen how to compute in practice the cocycle $\mu$
attached to $E$ and the cycle $c_\psi$ attached to an optimal embedding. The integration pairing then gives the Darmon point attached to $\psi$. That is to say, 
\begin{align}\label{eq: formula J_psi}
 J_\psi=\Xint\times\langle [\tilde c_\psi],[\mu]\rangle=\prod_{k=1}^C \Xint\times_{\P^1(\Q_p)}f_{D_k}(t) d\mu_{g_k}(t).
\end{align}
Each individual term $\Xint\times_{\P^1(\Q_p)}f_{D}(t) d\mu_{g}(t)$ can be numerically approximated by
a partial Riemann product, which for a covering  $\cU$ of $\P^1(\Q_p)$ is 
\[
 \prod_{U\in\cU}f_D(t_U)^{\mu_g(U)}, \ \  \ \ \text{$t_U$ any sample point in $U$}.
\]
Suppose that $D= \tau_2-\tau_1\in \Div^0\cH_p$, and that we want to compute the integral 
\begin{align*}
\Xint\times_{\mathbb{P}^1(\Q_p)}f_D(t)d\mu_g= \Xint\times_{\mathbb{P}^1(\Q_p)}\left(\frac{t-\tau_2}{t-\tau_1} \right)d\mu_g
\end{align*}
with an accuracy of $p^{-n}$. The size of the covering $\cU$ is determined by the affinoids in which $\tau_1$ and $\tau_2$ lie. To be more precise, let $r$ be a positive integer such that none of the elements $\tau_1,\tau_2,\omega_p\tau_1,\omega_p\tau_2$ is congruent to an integer modulo $p^r$. That is to say, such that
\begin{align}\label{eq:r}
  |\tau_1-i|_p> p^{-r}, \ |\tau_2-i|_p> p^{-r}, \ |\omega_p \tau_1-i|_p> p^{-r}, \ |\omega_p\tau_2-i|_p> p^{-r} \ \text{ for all } i\in \Z. 
\end{align}
Observe that we can find such an $r$ because $\tau_1,\tau_2,\omega_p\tau_1,\omega_p\tau_2$ do not belong to $\Q_p$. The function $f_D(t)$ is
locally constant modulo $p^n$ when restricted to open balls of
diameter
$p^{-(n+r)}$. Therefore, in order to obtain the value of $J_\psi$ correct modulo $p^n$ it is
enough to consider a finite covering $\cU_{n+r}$ of $\PP^1(\QQ_p)$  consisting of open balls of diameter $p^{-(n+r)}$.

Since $\mu_g$ is defined as an element of $\cF_0(\cE,\Z)\simeq\cF(\cE^+,\Z)$ it is useful to describe this covering of 
$\P^1(\Q_p)$ in terms of $\cE^+$ as follows. Note that
\[
\PP^1(\QQ_p)=\coprod_{t=0}^p \tilde\gamma_t^{-1}\ZZ_p,
\]
with $\tilde\gamma_t^{-1}\ZZ_p$ of diameter $1/p$. In Corollary~\ref{cor: opens in Z_p} we have described a 
covering $\cB(\ZZ_p,p^{-n})$ of $\ZZ_p$, and therefore one obtains the corresponding covering of $\PP^1(\QQ_p)$ as:
\[
\PP^1(\QQ_p)=\coprod_{t,i_m,j_m} (\tilde\gamma_{i_1}\gamma_{j_1}\cdots
\tilde\gamma_{i_n}\gamma_{j_n}\tilde\gamma_t)^{-1}\ZZ_p,
\]
where the indexes $i_m,j_m$ vary over $\{1,\dots,p\}$ and $t$ varies over $\{0,\ldots,p\}$.

\subsection{A numerical example}\label{subsection: example riemann}
We let $p = 13$, $D = 2\cdot 3$, and $M=1$.
Consider the elliptic curve with Cremona label ``78a1'':
\[
E\colon y^2 + xy = x^3 + x^2 - 19x + 685
\]

Let $K=\QQ(\sqrt{5})$, which is the quadratic field
with smallest discriminant satisfying that $2$, $3$ and $13$ are
inert in $K$. One observes that the point $P = (-2,12\sqrt{5}+1)\in E(K)$ generates the free part of $E(K)$.

Let $B$ be the quaternion algebra ramified precisely at $2$ and $3$.
It can be given as the $\QQ$-algebra $\QQ\langle i,j\rangle$, with relations $i^2 = 6$, $j^2= -1$, $ij = -ji$.

Let $\iota_{13}$ be the $\QQ$-algebra embedding of $B\to M_2(\QQ_{13})$ which sends:
\[
i\mapsto \mtx{0}{-1}{1}{0},\quad j\mapsto \frac{1}{\rho}\mtx{-1}{-24}{4}{1},
\]
with $\rho$ being the unique square root of $95$ in $\QQ_{13}$ which satisfies $\rho \equiv 2 \pmod{13}$. Let $R_0(1)\subset B$ be the maximal order with generators $\{1,i,(1+i+j)/2,(i+k)/2\}$, and let $\psi:\cO_K\hookrightarrow R_0(1)$ be the embedding that sends $\sqrt{5}\in K$ to $-i-j$. This
yields:
\[
\tau_\psi = (11  g + 9) + (12  g + 7) \cdot 13 + (12  g + 11) \cdot 13^{2} + (12  g + 12) \cdot 13^{3} + (12  g + 7) \cdot 13^{4} + O(13^5), 
\]
where $g\in K_{13}$ satisfies $g^2-g-1 = 0$, and $\gamma_\psi = (3-i-j)/2$.

The element $\gamma_\psi$ does not belong to the commutator subgroup of $\Gamma_0^6(1)_\text{ab}$, but $\gamma_\psi^{12}$ does. We rewrite the 
cycle
$\gamma_\psi^{12}\tns\tau_\psi$ in $H_1(\Gamma_0^6(1),\Div\cH_{13})$ as the sum of $16$ terms. Also, we act on
$\gamma_\psi^{12}$ with $t_5$.

Finally, we compute the integration pairing using Riemann products on coverings consisting of those opens of diameter
$13^{-n}$ for $n\in\{1,2,3\}$. Table~\ref{table:timings-riemann} gives the time that this computation took in our test computer.
Observe that the number of evaluations grows exponentially in $n$, and therefore so does the time it takes to complete
the integration.
\begin{table}[h]
\center
\begin{tabular}{rrr}
\toprule
$n$& Num. opens & Time (s)\\
\midrule
$1$ & $14$ & $3$\\
$2$ & $182$ & $49$\\
$3$ & $2366$ & $1158$\\
\bottomrule
\end{tabular}
\caption{Running time increases exponentially with the precision.}
\label{table:timings-riemann}
\end{table}

We obtain the value $J_\psi = (3g+2)13 + (g+9)13^2 + O(13^3)$ and, after applying the Tate parametrization, obtain $P_\psi\in E(K_{13})$ having coordinates
\[
(x,y) = (11 + 8\cdot 13 + 5\cdot 13^2 +O(13^3),(11g + 2) + (7g + 11)\cdot 13 + (7g + 12)\cdot 13^2 +O(13^3)).
\]
This point agrees with $48\cdot P$ up to the working precision of three
$13$-adic digits. Note that $48 = 12\cdot (5+1-a_5(E))$. The factor of $5+1-a_5(E)$ appears because of the application
of $t_5$, and the factor of $12$ appears because it was needed to kill the torsion of
$\Gamma_0^6(1)_\text{ab}$.

Although the previous computation gives evidence in support of the conjecture, the result is not very satisfying. 
Firstly, an approximation modulo $13^3$ could conceivably come from a numerical coincidence. More importantly, a 
previous knowledge of a generator for $E(K)$ was needed, and finding such a point is a hard problem in general. If we 
had a way to obtain a much better approximation, we could use algebraic recognition routines to guess the 
algebraic point. This is in fact the goal of the next section.

\section{The integration pairing via overconvergent cohomology}\label{sec: overconvergent}

We continue with the notation of \S\ref{subsection: computation riemann}. Namely, $\mu$ denotes the cohomology class associated to $E$ and $\tilde c_\psi$ the homology class associated to an optimal embedding $\psi$, which is of the form
\[\tilde c_{ \psi}=\sum_k g_k\otimes (\tau'_k-\tau_k) \]
for some $ g_k \in \Gamma$  and $ 
\tau_k,\tau_k'\in\cH_p$.
Therefore, the integrals involved in the computation of $J_{ \psi}$ are of the form
\begin{align}\label{eq: int I}
\Xint\times_{\P^1(\Q_p)}\left(\frac{t-\tau_2}{t-\tau_1} \right)d\mu_g(t),\ \text{ 
with } g\in\Gamma \text{ and } \tau_1,\tau_2\in\cH_p.
\end{align}
The goal of this section is to provide an algorithm for computing these integrals based on the overconvergent cohomology lifting theorems of \cite{pollack-pollack} which is more efficient than evaluating the Riemann products. In fact, the complexity of the overconvergent method that we present is polynomial in the number of $p$-adic digits of accuracy, whereas computing via Riemann sums is of exponential complexity.

 Since the type of integrals that can be directly computed by means of overconvergent cohomology are not exactly of the form \eqref{eq: int I}, we first need to perform certain transformations and reductions. Thus the method that we next describe can be divided into the following two steps:
\begin{enumerate}
\item Reduce the problem of computing integrals of the form \eqref{eq: int I} to that of computing the so-called \emph{moments} of $\mu$ at elements of $\GMp$. That is to say, express the integrals of \eqref{eq: int I} in terms of integrals of the form
  \begin{align}\label{eq: final ints}
    \int_{\Z_p} t^id\mu_g, \text{ for }g\in \GMp \text{  and  }i\in \Z_{\geq 0}.
  \end{align}
\item Give an algorithm for computing the integrals \eqref{eq: final ints} by means of the overconvergent cohomology lifting techniques of \cite{pollack-pollack}. 
\end{enumerate}
These two steps are explained in \S\ref{from int to mom} and \S\ref{sec: overconvergent cohomology}, respectively.
\subsection{From general integrals to moments}\label{from int to mom}
The first step in order to express integrals of the form \eqref{eq: int I} in terms of the moments \eqref{eq: final ints} is to consider covers of $\mathbb{P}^1(\Q_p)$ such that the integrand is analytic on each of the opens. Before fixing our choice of cover, we begin by proving a lemma that we will need in this process.
\begin{lemma}\label{lemma: previous lemma}
  Suppose that $\gamma\in \Gamma$ is of the form $\gamma=\tilde\gamma_{k_1}\gamma_{k_2}\tilde\gamma_{k_3}\cdots$ for some $k_\ell\in\{1,\dots,p\}$. Then $\mu_{\gamma|\Z_p}=0$ (i.e., the restriction of $\mu_\gamma$ to $\Z_p$ is $0$).
\end{lemma}
\begin{proof}
 It is enough to show that $\mu_{\gamma}(U_e)=0$ for all $U_e$ contained in $\Z_p$. By Corollary~\ref{cor: opens in Z_p} if $U_e=\gamma_e^{-1}\Z_p$ is contained in 
$ \Z_p$ then $\gamma_e=\tilde \gamma_{i_1}\gamma_{j_1}\cdots \tilde\gamma_{i_r}\gamma_{j_r}$ for some $i_s,j_s\in\{1,\dots,p\}$. Then we see that 
\begin{align*}
  \gamma_e \gamma = \tilde \gamma_{i_1}\gamma_{j_1}\cdots \tilde\gamma_{i_r}\gamma_{j_r} \tilde\gamma_{k_1}\gamma_{k_2}\tilde\gamma_{k_3}\cdots,
\end{align*}
from which we see that $\gamma_e\gamma$ belongs to our system of representatives $\mathcal{Y}$ for $\GMp\backslash \Gamma$. Therefore, from the identity $ \gamma_e\gamma = 1\cdot \gamma_e\gamma $
and the definition of $\mu$ (see \eqref{eq: definition of mu}) we obtain that $\mu_\gamma(U_e)=\varphi_1=0$.
\end{proof}
Let $r$ be a positive integer such that none of the elements $\tau_1,\tau_2,\omega_p\tau_1$, $\omega_p\tau_2$ is congruent to an integer modulo $p^r$, as in~\eqref{eq:r}. Consider a covering of $\P^1(\Q_p)$ of the form
\begin{align*}
  \mathbb{P}^1(\Q_p)=\bigsqcup_{t=0}^p\bigsqcup_{i_m,j_m}(\tilde\gamma_{i_1}\gamma_{j_1}\cdots \tilde\gamma_{i_n}\gamma_{j_n}\tilde\gamma_t)^{-1}\Z_p,
\end{align*}
with the $i_m,j_m$ varying over $\{1,\dots,p\}$, and such that every open has diameter  $\leq p^{-(r+1)}$.  Using this covering for breaking the integral \eqref{eq: int I} we are reduced to consider integrals of the form
\begin{align*}
  \Xint\times_{(\tilde\gamma_{i_1}\gamma_{j_1}\cdots \tilde\gamma_{i_n}\gamma_{j_n}\tilde\gamma_t)^{-1}\Z_p}\left(\frac{t-\tau_2}{t-\tau_1}\right)d\mu_g(t), \ \text{ for }g\in \Gamma
\end{align*}
and $t\in\{0,\dots,p\}$, $i_s,j_s\in\{1,\dots,p\}$.

To lighten the notation set $\alpha = \tilde\gamma_{i_1}\gamma_{j_1}\cdots \tilde\gamma_{i_n}\gamma_{j_n}\tilde\gamma_t$. Then we have that
\begin{align*}
  \Xint\times_{\alpha^{-1}\Z_p}\left(\frac{t-\tau_2}{t-\tau_1}\right)d\mu_g(t)=&  \Xint\times_{\Z_p}\left(\frac{\alpha^{-1}t-\tau_2}{\alpha^{-1}t-\tau_1}\right)d\mu_g(\alpha^{-1}t)=\Xint\times_{\Z_p}\left(\frac{\alpha^{-1}t-\tau_2}{\alpha^{-1}t-\tau_1}\right)d(\alpha \mu_{g})(t)\\
=& \Xint\times_{\Z_p}\left(\frac{\alpha^{-1}t-\tau_2}{\alpha^{-1}t-\tau_1}\right)d\mu_{\alpha g}(t) \div \Xint\times_{\Z_p}\left(\frac{\alpha^{-1}t-\tau_2}{\alpha^{-1}t-\tau_1}\right)d\mu_{\alpha}(t) \\ =& \Xint\times_{\Z_p}\left(\frac{\alpha^{-1}t-\tau_2}{\alpha^{-1}t-\tau_1}\right)d\mu_{\alpha g}(t), 
\end{align*}
where we have used the cocycle property of $\mu$ and the fact that $\mu_{\alpha|\Z_p}=0$ by Lemma \ref{lemma: previous lemma}. Therefore, letting $  \phi_0(t):=\left(\frac{\alpha^{-1}t-\tau_2}{\alpha^{-1}t-\tau_1}\right)$,  we have reduced the problem to compute integrals of the form
\begin{align}\label{eq: another int}
  \Xint\times_{\Z_p} \phi_0(t)d\mu_g, \text{ for } g \in \Gamma.
\end{align}
The next step is to express the above integrals in terms of integrals with respect to measures of the form $\mu_{g_0}$, where $g_0\in\GMp$. For instance, if we write $g=g_0\gamma$ with $g_0\in\GMp$ and $\gamma\in\mathcal{Y}$, Proposition \ref{prop: g in Gamma_0} below asserts that, under a certain condition on $\gamma$, we have an equality
\begin{align*}
  \Xint\times_{\Z_p}\phi_0(t)d\mu_{g}(t)=  \Xint\times_{\Z_p}\phi_0(t)d\mu_{g_0}(t).
\end{align*}
Recall that an edge $e\in\cE$ is said to be inward if $d(t(e),e_*)<d(s(e),e_*)$. Given $g\in\Gamma$  the edge $g^{-1}(e_*)$ is inward if and only if $g=g_0\gamma$ with $g_0\in\GMp$ and $\gamma\in\mathcal{Y}$ of the form
  \begin{align}\label{eq: condition gamma}
    \gamma=\tilde \gamma_t\gamma_{i_1}\tilde\gamma_{i_2}\cdots\ \  \text{ for some  } t,i_1,\dots,i_n\in\{1,\dots,p\}.
  \end{align}

\begin{proposition}\label{prop: g in Gamma_0}
  Let $g$ be an element in $ \Gamma$ such that $g^{-1}(e_*)$ is an inward edge. If $g=g_0\gamma$ with $\gamma$ as in \eqref{eq: condition gamma} then $\mu_{g|\Z_p}=\mu_{g_0|\Z_p}$.
\end{proposition}
\begin{proof}
By Lemma \ref{lemma: previous lemma} the measure $\mu_{\gamma}$ is $0$ when restricted to $\Z_p$. By the cocycle condition we have that $\mu_g= \mu_{g_0\gamma} =\mu_{g_0}+g_0\mu_{\gamma} $. Since $g_0\in \GMp$ if $U\subset\Z_p$ then $g_0^{-1}U\subset \Z_p$, so that $g_0\mu_\gamma(U)=\mu_\gamma(g_0^{-1}U)=0$ and we see that $g_0\mu_\gamma$ is $0$ when restricted to $\Z_p$.
\end{proof}
Suppose now that $g^{-1}(e_*)$ is outward, so that we can not directly apply Proposition \ref{prop: g in Gamma_0}. In this case observe that $(\tilde\gamma_ig)^{-1}(e_*)$ is inward for all $i\in\{1,\dots,p\}$. Thus we can write
\begin{align*}
  \Xint\times_{\Z_p}\phi_0(t)d\mu_g(t)&= \left(\Xint\times_{\mathbb{P}^1(\Q_p)\setminus\Z_p}\phi_0(t)d\mu_g(t)\right)^{-1}=\prod_{i=1}^p \left(\int_{\tilde\gamma_i^{-1}\Z_p}\phi_0(t)d\mu_g(t)\right)^{-1}\\ &= \prod_{i=1}^p \left(\int_{\Z_p}\phi_0(\tilde\gamma_i^{-1} t)d\mu_g(\tilde\gamma_i^{-1} t)\right)^{-1}=  \prod_{i=1}^p\left( \int_{\Z_p}\phi_0(\tilde\gamma_i^{-1} t)d\mu_{\tilde\gamma_i g}(t)\right)^{-1}
\end{align*}
and apply Proposition \ref{prop: g in Gamma_0} to each of the integrals in the last term. 

Summing up, we have expressed any integral as in \eqref{eq: int I} as a product of integrals of the form
\begin{align*}
  \Xint\times_{\Z_p}\phi_i(t)d\mu_g(t)\ \text{ for } g\in\GMp,
\end{align*}
where $\phi_i:=\phi_0(\tilde\gamma_i^{-1}t)$ for $i=0,1,\dots,p$.

Next, we show that the functions $\phi_i(t)$ are analytic on $\Z_p$, thanks to our choice of the covering of $\P^1(\Q_p)$. We begin by analyzing $\phi_0(t)$, since the result for the other $\phi_i(t)$ will follow easily from this case.

\begin{lemma}\label{lemma: the function is analytic}
 The function $\phi_0(t)=\frac{\alpha^{-1}t-\tau_2}{\alpha^{-1}t-\tau_1}$ is 
analytic on $\Z_p$ and has a series expansion of the form
  \begin{align}\label{eq:1}\displaystyle \frac{\alpha^{-1} t-\tau_2}{\alpha^{-1} t-\tau_1} = \alpha_0\left( 1+\sum_{n=1}^\infty \alpha_n p^{2n} t^n \right)\end{align} with the $\alpha_n$ belonging to $\cO_p$, the ring of integers of $K_p$, for all $n\geq 1$.
\end{lemma}
\begin{proof}
  Let $\mathfrak{J}=\{\smtx a b c d \in\GL_2(\Z_p)\colon p\mid c\}$,
which is the stabilizer of $\Z_p$ under the action $\GL_2(\Q_p)$ in the set of balls of $\P^1(\Q_p)$. Observe that if a function $\phi(t)$ satisfies the conclusions of the lemma, then also $\phi(\gamma t)$ does for all $\gamma\in \mathfrak{J}$. There are two cases to consider:
  \begin{enumerate}
  \item $\alpha^{-1}\Z_p$ is contained in $\Z_p$. Then, since $\alpha^{-1}\Z_p$ is a ball of diameter $ p^{-(r+1)}$ we find that
    \begin{align*}
      \alpha^{-1}\Z_p = \smtx{p^{r+1}}{i}{0}{1}\Z_p, \text{ for some } i\in\Z.
    \end{align*}
Therefore $\alpha^{-1}u_0=\smtx{p^{r+1}}{i}{0}{1}$ for some $u_0\in\mathfrak{J}$. Then, by our previous remark we can replace $t$ by $u_0 t$, and we find that
\begin{align*}
 \phi_0(u_0t)= \frac{\alpha^{-1}u_0 t-\tau_2}{\alpha^{-1}u_0 t-\tau_1} = \frac{\smtx{p^{r+1}}{i}{0}{1}t-\tau_2}{\smtx{p^{r+1}}{i}{0}{1} t-\tau_1}=\frac{(i-\tau_2)}{(i-\tau_1)}\frac{(1+\frac{p^{r+1}}{i-\tau_2}t)}{(1+\frac{p^{r+1}}{i-\tau_1}t)}.
\end{align*}
Now the key point is that by our choice of $r$ in \eqref{eq:r} we have that $v_p(i-\tau_2)< r$, so that $v_p\left(\frac{p^{r+1}}{i-\tau_j} \right)\geq 2$, and we result follows by taking the power series expansion in the above expression.
  \item $\alpha^{-1}\Z_p$ is contained in $\mathbb{P}^1(\Q_p)\setminus \Z_p$. In this case observe that $\omega_p\alpha^{-1}\Z_p\subset \Z_p$. Therefore 
    \begin{align*}
        \frac{\alpha^{-1} t-\tau_2}{\alpha^{-1} t-\tau_1} =   \frac{\omega_p\alpha^{-1} t-\omega_p\tau_2}{\omega_p\alpha^{-1} t-\omega_p\tau_1}, 
    \end{align*}
and the argument is exactly the same as before by noting that $\omega_p\alpha^{-1}\Z_p$ is of diameter $p^{-(r+1)}$ and therefore $\omega_p\alpha^{-1}=\smtx{p^{r+1}}{i}{0}{1}$ for some $i$, and that our choice of $r$ also works well for $\omega_p\tau_1$ and $\omega_p\tau_2$.
  \end{enumerate}
\end{proof}
\begin{proposition}\label{prop: the function is analytic}
For every $i=0,1,\dots,p$ the function $\phi_i(t)=\phi_0(\tilde\gamma_i^{-1}t)$ is 
analytic on $\Z_p$ and has a series expansion of the form
  \begin{align}\label{eq:2}\displaystyle \frac{\alpha^{-1} t-\tau_2}{\alpha^{-1} t-\tau_1} = \alpha_0\left( 1+\sum_{n=1}^\infty \alpha_n p^{n} t^n \right)\end{align} with the $\alpha_n$ belonging to $\cO_p$, the ring of integers of $K_p$, for all $n\geq 1$.
\end{proposition}
\begin{proof}
The result is clear for $i=0$. For $i>0$ observe that $\tilde\gamma_i$ is (up to an element in $\mathfrak{J}$) locally of the form $\smtx{-i}{1/p}{p}{0}$. Thus we can assume that 
\begin{align}
\phi_i(t)=\phi_0\left(\frac{-1/p}{pt-i}\right) \end{align}
and the result follows directly from Lemma \ref{lemma: the function is analytic} (note the factor $p^{2n}$ in the series expansion \eqref{eq:1}).
\end{proof}

At this point, we have reduced to compute integrals of the form
\begin{align}\label{eq: int I last red}
  I=\Xint\times_{\Z_p}\phi(t)d\mu_g(t), \text{ where } g\in\GMp \text{ and } \phi(t)=\alpha_0\left( 1+\sum_{n=1}^\infty \alpha_n p^{n} t^n \right).
\end{align}
Let $\log$ be the unique homomorphism $\log\colon K_p^\times\ra K_p$ such that $\log(1-t)=-\sum_{n=1}^\infty t^n/n$ and 
$\log(p)=0$. Its kernel is $p^\Z\times \mathbb{U}$, where $\mathbb{U}$ denotes the group of roots of unity in 
$K_p^\times$. Observe that the series of $\phi(t)$ converges for $t\in\Z_p$ and is constant modulo $p^{v_p(\alpha_0)+1}$. Thus the integral $I$ of \eqref{eq: int I last red}
can be computed as
\begin{align*}
 I= p^{v_p(\alpha_0)}\cdot \zeta \cdot \exp(\log I),
\end{align*}
where $\zeta$ is the Teichm\"uller lift of the unit part of $I$ modulo $p$, which can be computed as the Riemann product in the covering of $\Z_p$ by balls of diameter $p^{-1}$. Therefore, it only remains to compute the logarithm of $I$, which is the additive integral
\begin{align}\label{eq: additive int}
 \log I =\int_{\Z_p}\log \phi(t)d\mu_g(t).
\end{align}
Observe that $\log\phi(t)$ is analytic on $\Z_p$ and it has a series expansion of the form
\begin{align*}\label{eq: logarithm}
 \log\left(\phi(t)\right)=\beta_0+\sum_{n=1}^\infty \frac{\beta_n}{n} p^n t^n,\ \text{ with } \beta_i\in \cO_p.
\end{align*}
Let ${\mu_g}_{|_{\Z_p}}$ denote the measure on $\Z_p$ obtained by restriction of $\mu_g$, and let $\omega_g(n)$ denote its 
$n$-th moment:
\[
 \omega_g(n)=\int_{\Z_p}t^nd\mu_g(t).
\]
We see that the additive integral of \eqref{eq: additive int}
can be 
expressed as
\begin{align}\label{eq: integral in terms of moments}
 \beta_0\omega_g(0)+\sum_{n\geq 1}\frac{p^n}{n}\beta_n\omega_g(n)
\end{align}
for some $\beta_n\in\cO_p$. Now, suppose that we want to evaluate \eqref{eq: integral in terms of 
moments} modulo $p^M$; i.e., we want to compute the first $M$ $p$-adic digits of \eqref{eq: integral in 
terms of moments}. For this it is enough to compute, for each $i=0,1,\dots, M'$, the moment
$
 \omega_g(i)$ to an accuracy of $p^{M''-i}$, where
 \[
  M'=\sup\{n\colon \ord_p(p^n/n)<M\}\text{ and } \ M''=M+[\log(M')/\log(p)].
 \]
Summing up, we have reduced the problem of computing integrals as in \eqref{eq: int I} to that of
computing moments of the form 
\begin{align}\label{eq: the moments we need}
 \omega_g(i)=\int_{\Z_p}t^n d\mu_g(t)\pmod{p^{M''-i}}\ \ \text{ for $g\in\GMp$ and $i=0,\dots,M'$}.
\end{align}
In the next subsection we present an algorithm for computing the moments \eqref{eq: the moments we need} based on overconvergent cohomology.

\subsection{Computing the moments via overconvergent cohomology}\label{sec: overconvergent cohomology}
We present an algorithm for efficiently computing the moments $\omega_g(n)=\int_{\Z_p}t^nd\mu_g(t)$ for $g\in\GMp$,
based on
 the overconvergent cohomology methods of Pollack--Pollack \cite{pollack-pollack}. We begin by slightly adapting the
lifting
results
of~\cite[\S3]{pollack-pollack} (because we need to lift cocycles rather than just cohomology classes), and  then we will
show how to compute the moments $\mu$ by means of the lifted overconvergent cocycles.

Consider the module $\cD$ of locally-analytic $\Z_p$-valued
distributions on $\Z_p$. That is to say, given a distribution $\nu\in\cD$ and a  locally analytic function
$h:\Z_p\ra
\Z_p$ we have that $\nu(h(t))\in\Z_p$, and the map $h(t)\mapsto \nu(h(t))$ is linear and continuous.
Let $\Sigma_0(p)$ be the subsemigroup of $B^\times$
\[
 \Sigma_0(p) = \iota_p^{-1}\left(\{\smtx a b c d\in
\M_2(\Z_p)\colon c\equiv 0\pmod p, \ d\in\Z_p^\times,\
ad-bc\neq 0\}\right).
\]
It acts on the left on  $\cD$ as follows: if $h(t)$ is a locally analytic function on $\Z_p$ then
\[
(\gamma\cdot\nu)(h(t)) = \nu(h(\gamma\cdot t)),\quad \text{for } \nu\in\cD,\gamma\in\Sigma_0(p),
\]
where 
\[
 \gamma\cdot t=\frac{at+b}{ct+d}\ \ \text{ if } \iota_p(\gamma)=\mtx a b c d.
\]
The element $\pi\in
R_0(pM)$ defined in~\eqref{eq: def of s_i} lies in $\Sigma_0(p)$ and the double coset $\GMp\pi \GMp$ induces the
$U_p$-operator on the cocycles $Z^1(\GMp,\cD)$ and on  $H^1(\GMp,\cD)$. On cocycles, it is given explicitly by formula \eqref{eq:
def U_p}.

The module $\cD$ is equipped with the decreasing filtration
\[
\Fil^n\cD = \{\nu\in\cD\colon \nu(1) = 0, \quad \nu(t^{i})\in p^{n-i+1}\Z_p,\quad\forall i\geq 1\},
\]
which enjoys the following key properties:
\begin{lemma}
\begin{enumerate}
\item The natural projection $\cD\ra \displaystyle\varprojlim_n \cD/\Fil^n\cD$ is an isomorphism.
 \item If $\nu\in \Fil^n\cD$ then $\pi\cdot\nu\in\Fil^{n+1}\cD$.
\end{enumerate}
\end{lemma}
\begin{proof} If $\nu$ lies in $\Fil^n\cD$ for all $n$ then it is
necessarily the $0$ distribution, and this gives the first
property. As for the second, we shall see that $\pi\cdot \nu (t^i)=\nu(\pi\cdot t^i)$ lies in $p^{n-i+2}\Z_p$
whenever $\nu\in\Fil^n\cD$. Recall
that $\pi=\smtx p 0 0 1 u_\pi$ for some $u_\pi\in\Gamma_0^{\text{loc}}(p)$, say 
$u_\pi=\smtx{a}{b}{pc}{d}=\frac{1}{d}\smtx{a_0}{b_0}{pc_0}{1}$. Then
\[
 \pi\cdot t^i=p\left(u_\pi\cdot t^i\right)=\frac{a_0pt^i+b_0p}{c_0pt^i+1}=\sum_{j\geq 0}e_j(pt^i)^j,
\]
where the $e_j\in\Z_p$ arise from the series expansion of $1/(c_0pt^i+1)$.  Since $\nu\in
\Fil^n\cD$ we have that
\[
 \nu(1)=0 \ \text{ and } \ \nu(t^i)\in p^{n-i+1}\Z_p \ \text{ for all } i\geq 1,
\]
which implies that $\nu(\pi\cdot t^i)$ belongs to $p^{n-i+2}\Z_p$.
\end{proof}
Thanks to these two properties we are in the setting of \cite[\S3]{pollack-pollack}, in which very general
lifting theorems for cohomology classes hold. However, we will need the following slightly refined version of
\cite[Theorem 3.1]{pollack-pollack}, as we are interested in lifting cocycles rather than cohomology classes. 
\begin{proposition}\label{prop: analogous to PP}
 Let $\theta_0\in Z^r(\GMp,\cD/\Fil^0\cD)$ be an element such that $\opU_p
\theta_0=\alpha\theta_0$ for some
$\alpha\in \Z_p^\times$. Then there exists $\Theta\in Z^r(\GMp,\cD)$ such that:
\begin{enumerate}
 \item The image of $\Theta$ in $Z^r(\GMp,\cD/\Fil^0\cD)$ is equal to $\theta_0$ (i.e., $\Theta$ is a lift of
$\theta_0$), and
\item $\Theta$ is an eigen-cocycle for $\opU_p$ with eigenvalue $\alpha$.
\end{enumerate}
Moreover, if $\Theta'\in Z^r(\GMp,\cD)$ is another cocycle that lifts $\theta_0$ such that
$\opU_p\Theta'=\alpha\Theta'$ 
then $\Theta'=\Theta$.
\end{proposition}

As in \cite{pollack-pollack}, before proving this we state two lemmas that are, in fact, key to the proof.
\begin{lemma}\label{lemma: Up increases order}
 If $\theta\in C^r(\GMp,\Fil^n\cD)$ then $\opU_p \theta$ lies in $C^r(\GMp,\Fil^{n+1}\cD)$.
\end{lemma}
\begin{proof}
 This is identical to the proof of \cite[Lemma 3.3]{pollack-pollack}. 
\end{proof}

\begin{lemma}\label{lemma: kernel of reduction}
 If $\theta$ lies  in the kernel of $Z^r(\GMp,\cD)\ra Z^r(\GMp,\cD/\Fil^0\cD)$ and $\opU_p\theta=\alpha\theta$ for
some $\alpha\in \Z_p^\times$, then $\theta=0$.
\end{lemma}
\begin{proof}
 That $\theta$ lies in the kernel of $Z^r(\GMp,\cD)\ra Z^r(\GMp,\cD/\Fil^0\cD)$ is equivalent to the fact that
$\theta\in
Z^r(\GMp,\Fil^0\cD)$. Now $\theta= \alpha^{-1} \opU_p \theta$, and iterating this we find that
$\theta=\alpha^{-n}\opU_p^n \theta$. Thus by Lemma \ref{lemma: Up increases order} we see that $\theta$ lies in
$Z^r(\GMp,\Fil^n\cD)$ for all
$n$, and it must be $\theta=0$.
\end{proof}

\subsubsection*{Proof of proposition \ref{prop: analogous to PP}}
The proof is essentially the same as in
\cite{pollack-pollack}, but keeping track of the cocycles and not just the cohomology classes. It is
important to mention that the proof is actually constructive, and it provides with a very efficient method for
algorithmically computing such lifts. 

First we show the existence of $\Theta$. Let $\tilde\theta_0\in C^r(\GMp,\cD)$ be an arbitrary lift of $\theta_0$, and
for $n>0$
define $\tilde
\theta_n:=\alpha^{-n}\opU_p^n \tilde\theta_0$. Since $\theta_0$ is a cocycle and $\tilde\theta_0$ is a lift of
$\theta_0$, we have that $\partial^r\tilde\theta_0\in C^{r+1}(\GMp,\Fil^0\cD)$. Now 
\[
 \partial^r\tilde\theta_n=\alpha^{-n}\partial^r(\opU_p^n \tilde\theta_0)= \alpha^{-n}\opU_p^n
(\partial^r\tilde\theta_0);
\]
by Lemma \ref{lemma: Up increases order} this takes values in $\Fil^n\cD$. Let $\theta_n$ be the image of
$\tilde\theta_n$ in \[C^r(\GMp,\cD/\Fil^n\cD).\]
We have seen that, in fact, $\theta_n\in
Z^r(\GMp,\cD/\Fil^n\cD)$. Since
$\opU_p \theta_0=\alpha\theta_0$, we have that $\opU_p\tilde\theta_0-\alpha\tilde\theta_0$ belongs to
$C^r(\GMp,\Fil^0\cD)$. Therefore, one easily checks that $\opU_p \tilde \theta_n-\alpha\tilde\theta_n$ lies in
$C^r(\GMp,\Fil^n\cD)$, and we see that $\opU_p \theta_n=\alpha\theta_n$. Also, it is easy to see that
$\tilde\theta_n-\tilde\theta_{n-1}$ is in $C^r(\GMp,\Fil^n\cD)$. Then we can define $\Theta$ as
\[
\Theta=\{\theta_n\}\in \varprojlim Z^r(\GMp,\cD/\Fil^n\cD)=Z^r(\GMp,\cD).
\]
By construction $\Theta$ lifts $\theta_0$ and $\opU_p \Theta=\alpha\Theta$.

Now in order to prove uniqueness, let $\Theta'\in Z^n(\GMp,\cD)$ be an element that lifts $\theta_0$ and such that
$\opU_p\Theta'=\alpha\Theta'$. The difference $\Theta-\Theta'$ will be an element in the kernel of
\[Z^r(\GMp,\cD)\ra Z^r(\GMp,\cD/\Fil^0\cD)
\]
such that $\opU_p(\Theta-\Theta')=\alpha(\Theta-\Theta')$. By Lemma \ref{lemma: kernel of reduction} we
have that $\Theta-\Theta'=0$.

\qed

We will apply Proposition \ref{prop: analogous to PP} to the cocycle $\varphi=\varphi_E\in Z^1(\GMp,\Z)$ attached to
$E$ (and to a choice of sign at infinity) that we fixed in \S \ref{subsection: the cohomology class}. Indeed, since
\[
\Fil^0\cD = \{\nu\in\cD(\Z_p)\colon \nu(1)=0\}
\]
the map $\nu\mapsto \nu(1)$ induces an isomorphism
$
\cD/\Fil^0\cD\cong \Z_p.
$
Thus $\varphi$ can be naturally seen, after extending scalars to $\Z_p$, as a $1$-cocycle \[\varphi\in
Z^1(\GMp,\Z_p)=Z^1(\GMp,\cD/\Fil^0\cD).\]
Since $U_p\varphi=a_p\varphi$ with $a_p\in\{\pm 1\}$,  as a direct application of Proposition \ref{prop:
analogous to PP} we have:
\begin{proposition}\label{prop:uniquelifting}
 There exists a unique  $\Phi\in Z^1(\GMp,\cD)$ lifting $\varphi$ and such that $U_p\Phi=a_p\Phi$.
\end{proposition}
The proof of Proposition \ref{prop: analogous to PP} gives an effective method for computing (approximations to) $\Phi$:
one takes any cochain $\tilde \Phi$ in $C^1(\GMp,\cD)$ that lifts
$\varphi$, and iterates $a_p\opU_p$. After $k$ iterations, the natural
image of the resulting cochain $a_p^k\opU_p^k\tilde\Phi$ belongs to $Z^1(\GMp,\cD/\Fil^k\cD)$, and we can think of it
as an approximation to the desired $\Phi$, correct up to an element of $Z^1(\GMp,\Fil^k\cD)$. 

Let $g_1,\dots,g_t$ be the generators of $\GMp$, explicitly provided by Voight's algorithms
\cite{voight-fundamental}. If $g\in\GMp$ we can express $\Phi_g$ in terms of the $\Phi_{g_j}$ by means of the cocycle
relation of $\Phi$. A possible choice for $\tilde\Phi$ is then the chain determined by:
\[
 \tilde \Phi_{g_j}(1)=\varphi_{g_j}, \quad \tilde\Phi_{g_j}(t^i)=0 \quad \text{ for } i>0.
\]
The action of $a_pU_p$ is computed by means of formula \eqref{eq: def U_p}. After $k$ iterations only the values
$a_p\tilde \Phi_{g_j}(t^i)$ with $i\leq k$ will be different from $0$, and the resulting chain will be equal to $\Phi$
modulo $\Fil^k\cD$. Namely, we will have computed the quantities
\begin{align}\label{eq: the Phis that we compute}
 \Phi_{g_j}(t^i) \pmod{p^{k-i+1}},\quad i=0,\dots,k.
\end{align}

The next step is to show that $\Phi_g(t^i)$ is equal to the moment $\omega_g(i)$ for any $g\in\GMp$. This means that
the moments $\omega_g(i)\pmod{p^{k-i+1}}$ for $g\in\GMp$ can be computed by the method explained above.
\begin{proposition}
 Let $h$ be an analytic function on $\Z_p$ and $g\in\GMp$. Then
\[
\Phi_g(h(t)) = \int_{\Z_p} h(t)d\mu_g(t).
\]
\end{proposition}
\begin{proof}
 Let $\Psi$ be the cochain $\Psi\in C^1(\GMp,\cD)$ defined by the formula
\[
\Psi_g(h(t)) = \int_{\Z_p} h(t)d\mu_g(t).
\]
We will show that $\Psi$ is a cocycle, which lifts $\varphi$, and which satisfies $U_p\Psi=a_p\Psi$. This will finish the proof, because the uniqueness part of Proposition \ref{prop: analogous to PP} will imply that
$\Psi=\Phi$.

That $\Phi$ lifts $\varphi$ is an immediate consequence of property \ref{eq: property of m} of Theorem \ref{th: greenberg}. The cocycle property of $\mu$ implies that of $\Psi$:
\begin{align*}
  \Psi_{gh}(h(t)) &= \int_{\Z_p}h(t)d\mu_{gh}(t)=\int_{\Z_p}h(t)d(\mu_g(t) + \mu_h(g^{-1}t))\\
&=\int_{\Z_p} h(t)d\mu_g(t) + \int_{\Z_p} h(gt)d\mu_h(t)=\Psi_g(h(t)) + (g\cdot\Psi_h)(h(t)),
\end{align*}
where the second equality follows from a change of variables and the fact that $g^{-1}\ZZ_p = \ZZ_p$ for all $g\in\GMp$. As for the last claim, it follows from the computation:
\begin{align*}
(\opU_p\Psi)_g(h(t)) &= \sum_{i=1}^{p} \int_{\Z_p}h(s_it)d\mu_{t_i(g)}(t)\stackrel{(*)}{=}  \sum_{i=1}^{p} \int_{\Z_p}h(s_it)d\left(a_p \mu_g(s_it)\right)\\
&= a_p \sum_{i=1}^{p} \int_{s_i\Z_p}h(t)d\mu_g(t)=a_p\int_{\Z_p}h(t)d\mu_g(t)=a_p\Psi_g(h(t)),
\end{align*}
where the equality $(*)$ is justified by Lemma \ref{lemma:measure-key-result} below.
\end{proof}
We remark that Lemma \ref{lemma:measure-key-result}, although of a technical nature, provides the key calculation in the proof of the above proposition. Before proving it, we need two easy lemmas. 
\begin{lemma}\label{lemma: conjugating the gamma_e}
 Suppose $\gamma_e\in\cY$ is of the form $\gamma_e=\tilde\gamma_{i_1}\gamma_{j_1}\cdots
\tilde\gamma_{i_n}\gamma_{j_n}$ with all $i_k,j_k> 0$. Then
$\omega_p^{-1}\gamma_e\omega_p=\gamma_{i_1}\tilde\gamma_{j_1}\cdots
\gamma_{i_n}\tilde\gamma_{j_n}$.
\end{lemma}
\begin{proof}
 We will see it by induction. If $n=1$ we have that $\gamma_e=\tilde\gamma_{i_1}\gamma_{j_1}$,
and then
\[
 \omega_p^{-1}\gamma_e\omega_p=\omega_p^{-1}(p^{-1}\omega_p\gamma_{i_1}\omega_p)\gamma_{j_1}\omega_p=\gamma_{i_1
}p^{-1}\omega_p\gamma_{j_1}\omega_p=\gamma_{i_1}\tilde\gamma_{j_1}.
\]
For $n>1$ we write $\gamma_e=\gamma_{e'}\tilde\gamma_{i_n}\gamma_{j_n}$, where
$\gamma_{e'}=\tilde\gamma_{i_1}\gamma_{j_1}\cdots \tilde\gamma_{i_{n-1}}\gamma_{j_{n-1}}$. Then
\[
 \omega_p^{-1}\gamma_e\omega_p=(\omega_p^{-1}\gamma_{e'}\omega_p )(\omega_p^{-1}\tilde\gamma_{i_n}\gamma_{j_n}\omega_p)
\]
and now the result follows directly from the induction hypothesis.
\end{proof}

\begin{lemma}\label{lemma:property-cY}
 Let $\gamma_e\in\cY$ be such that $U_e\subseteq \Z_p$. Then $\omega_p^{-1}\gamma_e\omega_p\gamma_k$ belongs to
$\cY$ for all $k=0,\dots,p$.
\end{lemma}
\begin{proof}
 The statement is clear if $\gamma_e=1$. If $\gamma_e\neq 1$, then by Corollary \ref{cor: opens in Z_p} we have that
$\gamma_e$ is of the form
$\gamma_e=\tilde\gamma_{i_1}\gamma_{j_1}\cdots
\tilde\gamma_{i_n}\gamma_{j_n}$. Now by Lemma \ref{lemma: conjugating the gamma_e} we see that
\[
\omega_p^{-1}\gamma_e\omega_p\gamma_k=\gamma_{i_1}\tilde\gamma_{j_1}\cdots
\gamma_{i_n}\tilde\gamma_{j_n}\gamma_k,
\]
which clearly belongs to $\cY$.
\end{proof}
\begin{lemma}
\label{lemma:measure-key-result}
  Let  $g$ be an element in $\GMp$. For each $k=1,\ldots, p$ we have:
\begin{align}
 (\mu_{t_k(g)})_{|\Z_p}=(a_ps_k^{-1}\mu_g)_{|\Z_p};
\end{align}
that is to say, the measures $\mu_{t_k(g)}$ and $a_ps_k^{-1}\mu_g$ coincide when restricted to $\Z_p$.
\end{lemma}
\begin{proof}
It is enough to show that for every  $U_e\subset \Z_p$ one has
\begin{align}
\label{eq:measure}
\mu_{t_k(g)}(U_e)&=a_p\mu_g(s_kU_e).
\end{align}
Recall that $U_e=\gamma_e^{-1}\Z_p$ with $\gamma_e\in\cY$. By the definition of $\mu$ (see \eqref{eq: definition of mu})
 we have that
\[
\mu_{t_k(g)}(U_e)=\varphi_b,
\]
where $b\in\GMp$ is the element uniquely determined by the equation
\begin{align}
\label{eq:stareq}
 \gamma_et_k(g)&=b \gamma_{e'},\ \ \text{for some $\gamma_{e'}\in\cY$}.
\end{align}
Because of the definition of $t_k(g)$ (see \eqref{eq: def of t_k}) we have
\[
\gamma_et_k(g)=\gamma_e s_k^{-1}g s_{g\cdot k} = \gamma_e \omega_p \gamma_k g \gamma_{g\cdot k}^{-1} \omega_p^{-1},
\]
and combining this with \eqref{eq:stareq} we obtain
\begin{align}\label{eq: b}
\gamma_e \omega_p \gamma_k g = b \gamma_{e'} \omega_p \gamma_{g\cdot k}.
\end{align}

Now to calculate the right-hand side of~\eqref{eq:measure} we need to consider the open
\begin{align*}
s_kU_e &= s_k \gamma_e^{-1}\Z_p = \gamma_k^{-1} \omega_p^{-1} \gamma_e^{-1}\Z_p = \left(
\omega_p^{-1}\gamma_e\omega_p\gamma_k \right)^{-1}\omega_p^{-1}\Z_p\\
&= \PP^1(\Q_p)\setminus\left( \left(
\omega_p^{-1}\gamma_e\omega_p\gamma_k \right)^{-1}\Z_p\right).
\end{align*}
Therefore, the measure on the right-hand side of~\eqref{eq:measure} can be computed as
\begin{align}\label{eq: right-hand-side}
\mu_g(s_kU_e)=-\mu_g((\omega_p^{-1}\gamma_e\omega_p\gamma_k)^{-1}\Z_p).
\end{align}
Note that $\omega_p^{-1}\gamma_e\omega_p\gamma_k\in\cY$ thanks to Lemma~\ref{lemma:property-cY}, so 
in order to compute \eqref{eq: right-hand-side} we use \eqref{eq: b} to get the identity
\begin{align*}
\omega_p^{-1}\gamma_e\omega_p\gamma_k g &=  \omega_p^{-1} b \gamma_{e'} \omega_p \gamma_{g\cdot k}=(\omega_p^{-1} b\omega_p ) \omega_p^{-1} \gamma_{e'} \omega_p\gamma_{g\cdot k}.
\end{align*}
Now observe that $\gamma_{e'}^{-1}\Z_p\subset\Z_p$, so again Lemma~\ref{lemma:property-cY} gives that $\omega_p^{-1}
\gamma_{e'} \omega_p\gamma_{g\cdot k}\in\cY$ and we see that
\[
 \mu_g(s_kU_e)=-\varphi_{\omega_p^{-1} b \omega_p}=-(W_p\varphi)_b=(U_p\varphi)_b=a_p\varphi_b,
\]
and this concludes the proof.
\end{proof}

\section{Implementation and numerical evidence}\label{sec: numerical evidence}
We have implemented\footnote{The code is
available at \texttt{https://github.com/mmasdeu/darmonpoints}.} the algorithms of Sections \ref{sec: the algorithm} and \ref{sec: overconvergent} in
Sage~\cite{sage} and Magma \cite{magma}. Thanks to the overconvergent method we have been able to compute the integrals
up to a precision
of $p^{60}$, although one can easily reach much higher precision if needed. Recall the sample calculation of Section~\ref{subsection: example riemann}, which we have recalculated using the overconvergent method. In Table~\ref{table:timings-overconvergent} we list the time $t_1$ that it took to lift the original cocycle to the target precision $n$, and the time $t_2$ that it took to integrate the cycle to obtain $J_\tau$ with the target precision. One observes, as expected from the analysis carried out in~\cite{darmon-pollack} and which would easily carry over to our setting, that the complexity of the algorithms is polynomial (indeed quadratic). Note also that, while it took $1158$ seconds to obtain $3$ digits of precision using Riemann products, it took less than a third of this time to obtain $55$ digits of precision using the overconvergent method.
\begin{table}[h]
\begin{minipage}[t]{0.4\textwidth}
\centering
\begin{tabular}{rrrr}
\toprule
$n$& $t_1$ (s) & $t_2$ (s) &$t_1+t_2$ (s)\\
\midrule
5&	11&	6&	17\\
10&	24&	7&	31\\
15&	36&	9&	45\\
20&	55&	12&	67\\
25&	78&	16&	94\\
30&	108&	21&	129\\
35&	149&	26&	175\\
40&	191&	32&	223\\
45&	245&	39&	284\\
50&	307&	46&	353\\
55&	395&	55&	450\\
\bottomrule
\end{tabular}

\end{minipage}
\begin{minipage}{0.53\textwidth}
\includegraphics[width=\textwidth]{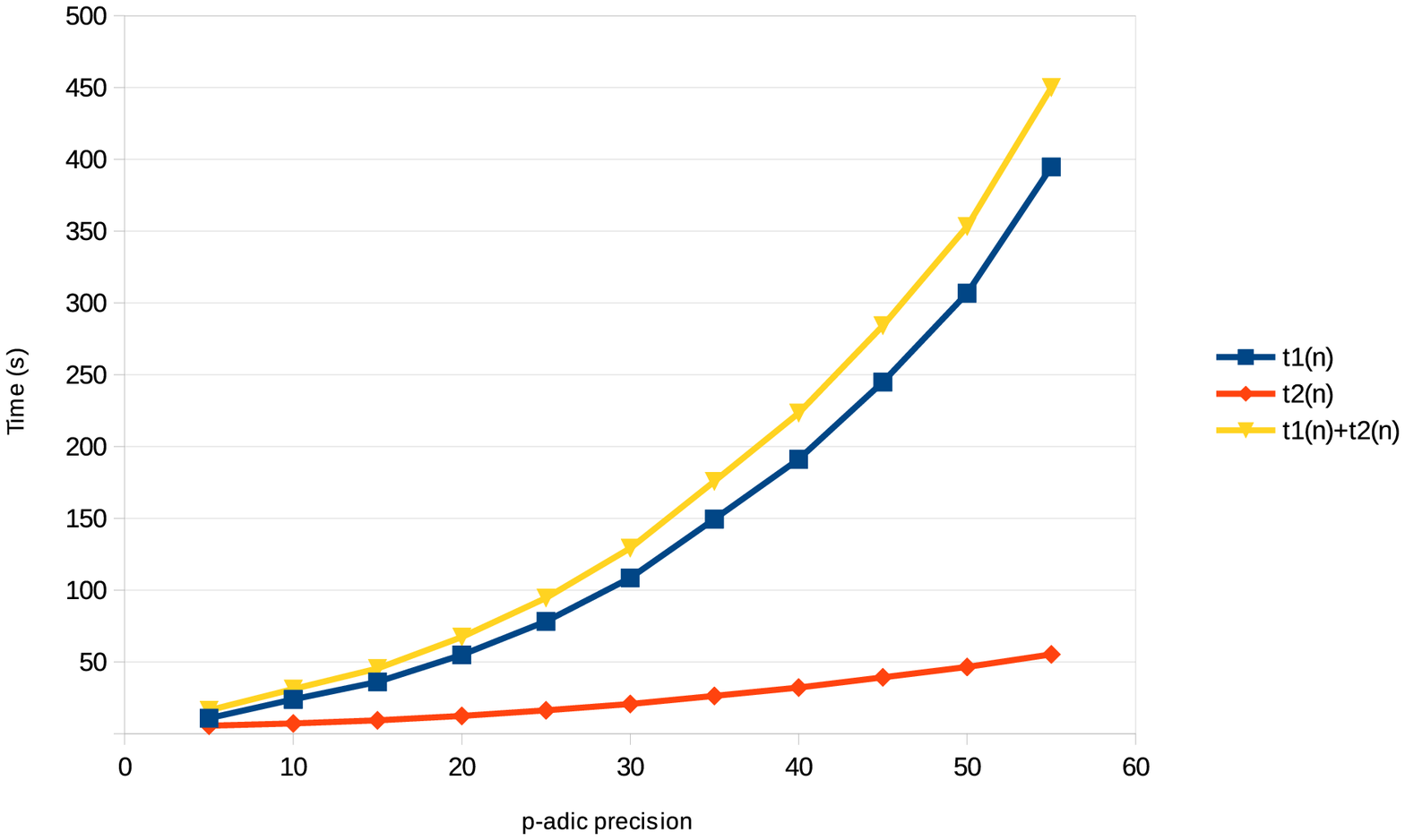}
\end{minipage}
\vspace{.4cm}
\caption{Running time increases sub-quadratically with the precision $n$.}
\label{table:timings-overconvergent}
\end{table}

Another salient feature of the overconvergent method is that one can regard the lifting of the cohomology class as a precomputation which depends only on the elliptic curve and the prime $p$. Note that, as the table indicates, this is what dominates the computing time. With this precomputation at hand, one can perform several integrals of different cycles (that is, yielding points attached to different real quadratic fields) with little extra effort. All this allows for a direct computation of rational
points, as opposite to the example of \S\ref{subsection: example riemann}, in which the low precision only permitted to
compare the computed Darmon point with an algebraic point previously found by naive search. 

Indeed, let $J_\tau\in K_p^\times$ be a Darmon point and let $P_\tau\in E(K_p)$ denote its image under Tate's
uniformization, whose coordinates conjecturally belong to a number field $H$. Using the algorithms described in this
article on can compute an approximation to $J_\tau$, and therefore to $P_\tau$. Then one can try to recognize its
coordinates as algebraic numbers via standard reconstruction techniques (see for instance \cite[\S1.6]{darmon-pollack}).
For this to work, the number of correct digits one needs to know of $J_\tau$ is roughly the height of $P_\tau$.

One difficulty that arises in this method is that $P_\tau$ is usually a multiple of the generator of
$E(H)/E(H)_{\text{tors}}$, say $P_\tau=n P_\tau'$. Therefore, $P_\tau$ might have very large height, even if
the generator $P_\tau'$ had small height. In this case it is easier to reconstruct $P_\tau'$, which has smaller height.
Note that $P_\tau'$ is the image under Tate's uniformization of an element of the form
\[
 J_\tau'=\zeta \exp\left(\frac{1}{n}{\log J_\tau}\right),
\]
where $\zeta$ is some Teichm\"uller representative in $K_p$. Therefore, since we can compute good approximations to 
$\log
J_\tau$, we can try to reconstruct $J_\tau'$ by trial and error on $\zeta$.

As a first example, consider the curve with Cremona label 78a1 with equation
\[
E:\quad y^2 + xy = x^3 + x^2 - 19x + 685.
\]
Table~\ref{table:curve78a1} lists points on $E(\QQ(\sqrt{d_K}))$ for those discriminants $d_K< 600$ in which $2$, $3$ 
and $13$ are inert and such that $K=\QQ(\sqrt{d_K})$ has class number one. They are computed using the plus character 
$\lambda_E^+$ and optimal embeddings of the maximal order $\cO_K$. Observe that the points are defined over $K$ rather than over abelian extensions, since the class number is one.

Table \ref{table:curve110a1} lists similar computations for the curve with Cremona label 110a1 and equation
\[
 E\colon y^2 + xy + y = x^3 + x^2 + 10x - 45.
\]
Observe that some of the points, e.g. the one over $\QQ(\sqrt{237})$, could have not been found by naive search 
methods due to their height. Table \ref{table:curve110a1bis} shows the same 
points computed with the different factorization of the conductor $110$, namely $p = 11$ and $D=10$. Note that for 
$d_K=277$ we were not able to recognize the point. This is probably due to the fact that the working precision 
($p^{60}$ in this case) is lower, since $p=5$ instead of $p=11$. In these two cases the points obtained are twice the 
expected multiple of the generator. Table~\ref{table:curve114a1} is another example with $D=6$, but in this case some of the points obtained (note e.g. $d_K=269$) have considerable height.

The examples shown above have in common that the group $H_1(\Gamma_0^D(M),\Z)$ is finite. Although our algorithms do not require this condition to be true, the implementation is greatly simplified in this case. However, our implementation works in a broader range of cases. As an example, we have computed an example with $D=15$, where the above group has $\Z$-rank $1$: consider the elliptic curve with Cremona label 285c1 ($285=19\cdot 15$) given by the equation
\[
E\colon y^2 + xy = x^3 + x^2 + 23x -176.
\]
Working with $p=19$ and precision $19^{60}$, our algorithm has been able to recover the point:
\[
P=\left(\frac{372503}{60543} , \frac{60805639}{78826986} \sqrt{413} - \frac{372503}{121086} \right)\in E\left(\QQ(\sqrt{413})\right).
\]
\renewcommand{\arraystretch}{1.4}
\begin{small}
\begin{table}[H]
\begin{tabular}{c|c}
$d_K$  & $P$\\\hline
$5$  & $1\cdot 48\cdot\left(-2, 12\sqrt{5}+1\right)$ \\
$149$  & $1\cdot48\cdot\left(1558,-5040\sqrt{149}-779\right)$ \\
$197$  & $1\cdot 48\cdot\left(\frac{310}{49},\frac{720}{343}\sqrt{197} - \frac{155}{49} \right)$\\
$293$  & $1\cdot 48\cdot\left(40,-15\sqrt{293}-20\right)$\\
$317$  & $1\cdot 48\cdot\left(382,-420\sqrt{317}-191\right)$\\
$437$  & $1\cdot 48\cdot\left(\frac{986}{23}, \frac{7200}{529}\sqrt{437} - \frac{493}{23}\right)$\\
$461$  & $1\cdot 48 \cdot\left(232,-165\sqrt{461}-116\right)$\\
$509$  & $1\cdot 48\cdot\left(-\frac{2}{289},-\frac{5700}{4913}\sqrt{509}+\frac{1}{289}\right)$\\
$557$  & $1\cdot 48\cdot\left(\frac{75622}{121} , \frac{882000}{1331}\sqrt{557} - \frac{37811}{121}\right)$\\
\end{tabular}

\caption{Darmon points on curve 78a1 with $p=13$ and $D=6$.}\label{table:curve78a1}
\end{table}
\begin{table}[H]
\begin{tabular}{c|c}
$d_K$  & $P$\\\hline
$13$  & $2\cdot 30\cdot\left(\frac{1103}{81}-\frac{250}{81}\sqrt{13} ,-\frac{52403}{729} 
+\frac{13750}{729}\sqrt{13} \right)$ \\
$173$  & $2\cdot 30\cdot\left(\frac{1532132}{9025} ,-\frac{1541157}{18050}-\frac{289481483}{1714750}\sqrt{173} 
\right)$ \\
$237$  & $2\cdot 30\cdot \big(\frac{190966548837842073867}{4016648659658412649} 
-\frac{10722443619184119320}{4016648659658412649}\sqrt{237},$\\
& $-\frac{3505590193011437142853233857149}{8049997913829845411423756107}+\frac{
235448460130564520991320372200}{8049997913829845411423756107}\sqrt{237} 
\big )$ \\
$277$  & $2\cdot 30\left(\frac{46317716623881}{12553387541776} 
,-\frac{58871104165657}{25106775083552}-\frac{20912769335239055243}{44477606117965542976}\sqrt{277} 
\right)$ \\
$293$  & $2\cdot 30\cdot\left(\frac{7088486530742}{2971834657801} 
,-\frac{10060321188543}{5943669315602}-\frac{591566427769149607}{10246297476835603402}\sqrt{293} 
\right)$ \\
$373$  & $2\cdot 30 \cdot\left(\frac{298780258398}{62087183929} 
,-\frac{360867442327}{124174367858}-\frac{19368919551426449}{30940899762281434}\sqrt{373} 
\right)$ \\

\end{tabular}

\caption{Darmon points on curve 110a1 with $p=11$ and $D=10$.}\label{table:curve110a1}
\end{table}
\begin{table}[H]
\begin{tabular}{c|c}
$d_K$  & $P$\\\hline
$13$  & $2\cdot 12\cdot\left(4 , \frac{5}{2}\sqrt{13} - \frac{5}{2}\right)$\\
$173$  & $2\cdot 12\cdot\left(\frac{1532132}{9025}, -\frac{289481483}{1714750}\sqrt{173} - \frac{1541157}{18050}\right)$ \\
$237$  & $2\cdot 12\cdot\left(\frac{5585462179}{1193768112} , -\frac{53751973226309}{71439858894528}\sqrt{237} - \frac{6779230291}{2387536224} \right)$ \\
$277$  & ---\\
$293$  & $2\cdot 12\cdot\left(\frac{7088486530742}{2971834657801}, -\frac{591566427769149607}{10246297476835603402}\sqrt{293} - \frac{10060321188543}{5943669315602}\right)$\\
$373$  & $2\cdot 12 \cdot\left(\frac{298780258398}{62087183929} , \frac{19368919551426449}{30940899762281434}\sqrt{373} - \frac{360867442327}{124174367858}\right)$
\end{tabular}

\caption{Darmon points on curve 110a1 with $p=5$ and $D=22$.}\label{table:curve110a1bis}
\end{table}
\begin{table}[H]
\begin{tabular}{c|c}
$d_K$  & $P$\\\hline
$29$  & $1 \cdot 72 \cdot \left(-\frac{6}{25}\sqrt{29} - \frac{38}{25}, -\frac{18}{125}\sqrt{29} + \frac{86}{125}\right)$ \\
$53$  & $1 \cdot 72 \cdot \left(-\frac{1}{9}, \frac{7}{54}\sqrt{53} + \frac{1}{18} \right)$ \\
$173$  & $1 \cdot 72 \cdot \left(-\frac{3481}{13689}, \frac{347333}{3203226}\sqrt{173} + \frac{3481}{27378}\right)$\\
$269$  & $1 \cdot 72 \cdot \Big(\frac{1647149414400}{23887470525361}\sqrt{269} - \frac{43248475603556}{23887470525361},$\\
&$ \frac{2359447648611379200}{116749558330761905641}\sqrt{269} + \frac{268177497417024307564}{116749558330761905641}\Big)$\\
$293$  & $1 \cdot 72 \cdot \left(\frac{21289143620808}{4902225525409}, \frac{4567039561444642548}{10854002829131490673}\sqrt{293} - \frac{10644571810404}{4902225525409}\right)$\\
$317$  & $1 \cdot 72 \cdot \left(-\frac{25}{9}, -\frac{5}{54}\sqrt{317} + \frac{25}{18}\right)$\\
$341$  & $1 \cdot 72 \cdot \left(\frac{3449809443179}{499880896975},\frac{3600393040902501011}{3935597293546963250}\sqrt{341} - \frac{3449809443179}{999761793950}\right)$\\
$413$  & $1 \cdot 72 \cdot \left(\frac{59}{7} ,\frac{113}{98}\sqrt{413} - \frac{59}{14}\right)$
\end{tabular}

\caption{Darmon points on curve 114a1 with $p=19$ and $D=6$.}\label{table:curve114a1}
\end{table}
\end{small}
\clearpage
\bibliographystyle{amsalpha}
\bibliography{refs}
\end{document}